\documentclass[11pt]{article}%
\usepackage{amsfonts}
\usepackage{amssymb}
\usepackage{amsmath}
\usepackage{graphicx}
\usepackage{theorem}
\usepackage[a4paper]{geometry}
\usepackage[toc]{appendix}
\usepackage{hyperref}
\usepackage{makeidx}%
\providecommand{\U}[1]{\protect\rule{.1in}{.1in}}
\theoremstyle{break}
\newtheorem{theorem}{Theorem}[section]

\newtheorem{condition}[theorem]{Condition}

\newtheorem{lemma}[theorem]{Lemma}

\newtheorem{proposition}[theorem]{Proposition}
\theorembodyfont{\upshape}
\newtheorem{remark}[theorem]{Remark}

\newenvironment{proof}[1][Proof]{\textbf{#1.} }{\ \rule{0.5em}{0.5em}}
\numberwithin{equation}{section}
\makeindex
\begin{document}

\title{A local CLT for convolution equations with an application to weakly
self-avoiding random walks\thanks{
words and phrases: central limit theorem, convolution equations, self-avoiding
random walks.}}
\author{Luca Avena\thanks{supported by the Swiss National Science Foundation under
contract 138141, and by the Forschungskredit of the University of Z\"{u}rich.}
\and Erwin Bolthausen\thanks{supported by the Swiss National Science Foundation
under contract 138141, and by the Humboldt Foundation.}
\and Christine Ritzmann
\and University of Z\"{u}rich}
\date{}
\maketitle

\begin{abstract}
We prove error bounds in a central limit theorem for solutions of certain
convolution equations. The main motivation for investigating these equations
stems from applications to lace expansions, in particular to weakly
self-avoiding random walks in high dimensions. As an application we treat such
self-avoiding walks in continuous space. The bounds obtained are sharper than
the ones obtained by other methods.

\end{abstract}

\section{Introduction}

\subsection{On some convolution equations}

Let $\phi$ be the standard normal density in $\mathbb{R}^{d},$ $\mathbf{B}%
=\left\{  B_{k}\right\}  _{k\geq1}$ be a sequence of rotationally invariant
integrable functions, and $\lambda>0$ a (small) parameter. Define recursively%
\begin{align}
C_{0}  &  =\delta_{0},\nonumber\\
C_{n}  &  =C_{n-1}\ast\phi+\lambda\sum_{k=1}^{n}c_{k}B_{k}\ast C_{n-k}%
,\ n\geq1, \label{Basic_ConvolEqu}%
\end{align}
where%
\[
c_{n}\overset{\mathrm{def}}{=}\int C_{n}\left(  x\right)  dx.
\]
$\delta_{0}$ denotes the Dirac \textquotedblleft function\textquotedblright.
%All our (signed) distributions will have densities, except those at index $0.$

As written above, the sequence $\mathbf{C}=\left\{  C_{n}\right\}  _{n\geq0}$
is not quite recursively defined as the right hand side in
(\ref{Basic_ConvolEqu}) contains the summand $c_{n}B_{n}.$ The sequence
$\left\{  c_{n}\right\}  $ itself satisfies%
\begin{align}
c_{0}  &  =1,\label{Equ_sequ_c}\\
c_{n}  &  =c_{n-1}+\lambda\sum_{k=1}^{n}c_{k}b_{k}c_{n-k},\ n\geq1,\nonumber
\end{align}
where $b_{k}=\int B_{k}\left(  x\right)  dx.$ Therefore, if $\lambda\left\vert
b_{n}\right\vert <1$ for all $n,$ these equations define the sequence
$\left\{  c_{n}\right\}  $ uniquely, and then also $\mathbf{C}$ is well
defined. We will always assume that we are in this situation.

The main assumption is a decay property of the $B_{n}$ for large $n.$ We will
also assume Gaussian decay properties in space which are natural for the
applications to self-avoiding walks we have in mind. The method we present
here can probably be adapted to treat situations with less severe decay
assumptions in space, but we have not worked that out.

Our main interest is to prove a local central limit theorem for the signed
density $C_{n}/c_{n}$ under appropriate conditions on $\mathbf{B}$ and
$\lambda.$ Of course, the parameter $\lambda$ can be incorporated into
$\mathbf{B}$. However, the approach we follow is purely perturbative. We will
give conditions on $\mathbf{B},$ and then state that if in addition $\lambda$
is small enough a CLT holds.

At the expense of a few complications, we could also investigate the case
where the first summand in \eqref{Basic_ConvolEqu} is $C_{n-1}\ast S$ with a
rotationally invariant density $S.$ We however feel that this generalization
would somehow obscure the main line of the argument. To step out from the
rotationally invariant case leads however to new, complicated, and interesting
problems which will be presented elsewhere.

The main motivation for our investigation comes from Weakly Self-Avoiding
Random Walks (WSAW). Indeed, as we will show, by using the so called
\emph{lace expansion}, WSAW satisfy an equation as in \eqref{Basic_ConvolEqu}.

In the next setion we state our main theorem on this type of convolution
equation, Theorem \ref{Th_main}. In Section \ref{DefContWSAW}, we introduce
WSAW in continuous space and state a local CLT, Theorem \ref{Th_main_SAW},
that will be deduced from Theorem \ref{Th_main}. To conclude this introductory
part, in Subsecton \ref{Literature} we discuss how this work relates to the
existent literature and we describe the structure of the paper.

\subsection{Main result on convolution equations}

\label{resultOnConvEqu} Before stating our general result on convolution
equations as in \eqref{Basic_ConvolEqu}, we first fix some notations and
define the set of conditions we need for the $B_{k}$'s in \eqref{Basic_ConvolEqu}.

$\mathbb{N}$ is the set of natural numbers $\left\{  1,2,\ldots,\right\}  $
and $\mathbb{N}_{0}\overset{\mathrm{def}}{=}\mathbb{N\cup}\left\{  0\right\}
.$ For $t>0,$ $\phi_{t}$ is the centered normal density in $\mathbb{R}^{d}$
with covariance matrix $t\times\mathrm{identity}.$ We write $\phi$ for
$\phi_{1}.$

We write $\mathcal{C}_{\ast}\left(  \mathbb{R}^{d}\right)  $ for the set of
continuous, integrable functions $f:\mathbb{R}^{d}\rightarrow\mathbb{R},$
vanishing at $\infty,$ which are of the form $f\left(  x\right)  =f_{0}\left(
\left\vert x\right\vert \right)  $ for some continuous function $f_{0}%
:[0,\infty)\rightarrow\mathbb{R}.$ We also write $\mathcal{C}_{\ast}%
^{+}\left(  \mathbb{R}^{d}\right)  $ for the strictly positive ones.

Here are the conditions we need for $\mathbf{B}$:

\begin{condition}
[Decay assumptions on $\mathbf{B}$-sequence]\label{Cond_Main}Assume that the
functions $B_{m}\in\mathcal{C}_{\ast}\left(  \mathbb{R}^{d}\right)  $ in
\eqref{Basic_ConvolEqu} are dominated in absolute value by functions
$\Gamma_{m}\in\mathcal{C}_{\ast}^{+}\left(  \mathbb{R}^{d}\right)  $ which
satisfy the following conditions:

\begin{enumerate}
\item[B1] There exist numbers $\chi_{n}\left(  s\right)  >0,\ 1\leq s\leq n,$
satisfying $\chi_{n}\left(  s\right)  =\chi_{n}\left(  n-s\right)  ,$ and for
some constant $K_{1}$%
\begin{equation}
\sum_{s=1}^{n-1}\left(  s\wedge\left(  n-s\right)  \right)  \chi_{n}\left(
s\right)  \leq K_{1},\ \forall n, \label{Sequ1}%
\end{equation}
such that%
\begin{equation}
\Gamma_{m}\ast\Gamma_{n}\leq\chi_{m+n}\left(  m\right)  \Gamma_{n+m},
\label{Gamma_Convol1}%
\end{equation}

\item[B2] There exists a constant $K_{2}>0$ such that for $t\leq s\leq2t$ one
has%
\begin{equation}
\Gamma_{s}\leq K_{2}\Gamma_{2t} \label{Gamma_Domination}%
\end{equation}

\item[B3] There exists $K_{3}>0$ such that for $m\leq t,\ m\in\mathbb{N},$
$t\in\mathbb{R}^{+},$ $k=0,1,2,$ one has%
\begin{equation}
\int\phi_{t}\left(  x-y\right)  \left\vert y\right\vert ^{2k}\Gamma_{m}\left(
y\right)  dy\leq K_{3}\gamma_{m}^{\left(  k\right)  }\phi_{t+m}\left(
x\right)  , \label{Gamma_Convol3}%
\end{equation}
where%
\[
\gamma_{m}^{\left(  k\right)  }\overset{\mathrm{def}}{=}\int\left\vert
y\right\vert ^{2k}\Gamma_{m}\left(  y\right)  dy.
\]

\item[B4] The three sequences $\left\{  \gamma_{n}^{\left(  i\right)
}\right\}  _{n\in\mathbb{N}},\ i=0,1,2$, are non-increasing, and%
\begin{equation}
K_{4}\overset{\mathrm{def}}{=}\sum_{n}n\gamma_{n}^{\left(  0\right)  }%
<\infty,\ K_{5}\overset{\mathrm{def}}{=}\sum_{n}\gamma_{n}^{\left(  1\right)
}<\infty,\ K_{6}\overset{\mathrm{def}}{=}\sum_{n}n^{-1}\gamma_{n}^{\left(
2\right)  }<\infty. \label{Gamma_Convol4}%
\end{equation}

\end{enumerate}
\end{condition}

A simple example where the conditions B1-B4 are satisfied is $\Gamma
_{n}=n^{-a}\phi_{n/2},\ a>2,$ but the application to self-avoiding walks needs
a slightly more complicated choice, as will be discussed later.

We will often write $\gamma_{m}$ for $\gamma_{m}^{\left(  0\right)  }.$

We remark that under the above condition, one has for%
\begin{equation}
\label{bm}b_{n}\overset{\mathrm{def}}{=}\int B_{n}\left(  x\right)  dx
\end{equation}
the estimate%
\[
\left\vert b_{n}\right\vert \leq\gamma_{n}%
\]
with%
\begin{equation}
\gamma_{m}\gamma_{n}\leq\chi_{m+n}\left(  m\right)  \gamma_{n+m}.
\label{bound_gamma}%
\end{equation}

Next, fix an arbitrary positive $\varepsilon>0,$ and write%
\begin{equation}
\psi_{n}\overset{\mathrm{def}}{=}\phi_{n\delta\left(  1+\varepsilon\right)  },
\label{Def_psi_n}%
\end{equation}
with $\delta$ defined below in (\ref{Def_delta}).

In the sequel, we will use $L$ as a positive constant, not necessarily the
same at different occurrences, which may depend on $d,\varepsilon,K_{1}-K_{6}%
$, but not on $n,\lambda$.

Let%
\begin{equation}
\zeta_{n}^{\left(  1\right)  }\overset{\mathrm{def}}{=}1+\sum_{i=0}^{2}%
\sum_{m=1}^{n}m^{2-i}\gamma_{m}^{\left(  i\right)  } \label{Def_zeta}%
\end{equation}%
\[
\zeta_{n}^{\left(  2\right)  }\overset{\mathrm{def}}{=}\sum_{m=n}^{\infty
}\left(  \gamma_{m}^{\left(  1\right)  }+m\gamma_{m}\right)  ,
\]%
\[
\overline{\zeta}_{n}\overset{\mathrm{def}}{=}n^{-2}\sum_{j=1}^{n}\zeta
_{j}^{\left(  1\right)  }+n^{-1}\sum_{j=1}^{n}\zeta_{j}^{\left(  2\right)  }.
\]
Because of (\ref{Gamma_Domination}) and (\ref{Gamma_Convol4}) we have%
\begin{equation}
\lim_{n\rightarrow\infty}\overline{\zeta}_{n}=0,\ \sum_{n}n^{-1}%
\overline{\zeta}_{n}<\infty,\ \overline{\zeta}_{m}\leq\overline{\zeta}%
_{2n}\ \mathrm{for\ }n\leq m\leq2n. \label{Bound_zetabar}%
\end{equation}

Remark that%
\begin{equation}
\overline{\zeta}_{n}\geq\frac{1}{n^{2}}\sum_{j=1}^{n}\sum_{m=1}^{n}m^{2}%
\gamma_{m}\geq\frac{n^{2}}{L}\gamma_{n}. \label{Bound_Gamman_by_zetabar}%
\end{equation}

We can finally state our main theorem on convolution equations:

\begin{theorem}
[Local CLT for convolution equations]\label{Th_main} Assume Condition
\ref{Cond_Main}. Then, if $\lambda$ is small enough (depending on
$d,\varepsilon$ and $K_{1}$-$K_{6}$), the following estimates holds%
\begin{equation}
\left\vert C_{n}\left(  x\right)  /c_{n}-\phi_{n\delta}\left(  x\right)
\right\vert \leq L\lambda\left[  \sum_{s=1}^{\left[  n/2\right]  }s\left(
\psi_{s}\ast\Gamma_{n-s}\right)  \left(  x\right)  +\overline{\zeta}_{n}%
\psi_{n}\left(  x\right)  \right]  , \label{mainbound}%
\end{equation}
where $\delta=\delta\left(  \mathbf{B},\lambda\right)  >0$ is defined in
(\ref{Def_delta}) below.
\end{theorem}

In the example $\Gamma_{n}\left(  x\right)  =n^{-a}\phi_{n/2}\left(  x\right)
,\ 2<a<3,$ one has $\zeta_{n}^{\left(  1\right)  }=\operatorname*{const}\times
n^{3-a},$ $\zeta_{n}^{\left(  2\right)  }=\operatorname*{const}\times
n^{2-a},$ and therefore $\overline{\zeta}_{n}=\operatorname*{const}\times
n^{2-a},$ and thus%
\[
\left\vert C_{n}\left(  x\right)  /c_{n}-\phi_{n\delta}\left(  x\right)
\right\vert \leq L\lambda n^{2-a}\psi_{n}%
\]
giving a local CLT with a precise error estimate. For $a>3,$ we get%
\[
\left\vert C_{n}\left(  x\right)  /c_{n}-\phi_{n\delta}\left(  x\right)
\right\vert \leq L\lambda n^{-1}\psi_{n}.
\]
As remarked above, this $\Gamma_{n}$ cannot work for the application to
self-avoiding walks, and in fact, a pure local CLT is not possible in that case.

\subsection{WSAW on $\mathbb{R}^{d}$ and result}

\label{DefContWSAW} The main motivation for our investigation of these type of
convolution equations comes from WSAW as first investigated by Brydges and
Spencer in the seminal paper \cite{BrSp}. Their results are for random walks
on the $d$-dimensional lattice $\mathbb{Z}^{d},\ d\geq5.$ In contrast, we now
introduce and investigate weakly self-avoiding random walks on $\mathbb{R}%
^{d}$ with standard normal increments. The model has two parameters
$\lambda,\rho>0,$ $\rho$ being the range of the interaction, and $\lambda$ the
strength. We set $\mathbb{I}_{\rho}\left(  x\right)  \overset{\mathrm{def}}%
{=}1_{\left\{  \left\vert x\right\vert \leq\rho\right\}  }$, and if
$\mathbf{x}=\left(  x_{1},\ldots,x_{n}\right)  \in\left(  \mathbb{R}%
^{d}\right)  ^{n},$ and $0\leq i<j\leq n,$ we set $U_{ij}^{\rho}\left(
\mathbf{x}\right)  \overset{\mathrm{def}}{=}\mathbb{I}_{\rho}\left(
x_{j}-x_{i}\right)  $, where $x_{0}=0.$ Then, for $0\leq\lambda\leq1,$ define
the probability measure $P_{n,\lambda,\rho}$ on $\left(  \mathbb{R}%
^{d}\right)  ^{n}$ by its density with respect to Lebesgue measure:%
\begin{equation}
p_{n,\lambda,\rho}\left(  \mathbf{x}\right)  =\frac{1}{Z_{n,\lambda,\rho}%
}K_{\lambda,\rho}\left[  0,n\right]  \left(  \mathbf{x}\right)  \Phi\left[
0,n\right]  \left(  \mathbf{x}\right)  , \label{fundamental_formula}%
\end{equation}
where%
\begin{gather}
K_{\lambda,\rho}\left[  a,b\right]  \left(  \mathbf{x}\right)  \overset
{\mathrm{def}}{=}\prod_{a\leq i<j\leq b}\left(  1-\lambda U_{ij}^{\rho}\left(
\mathbf{x}\right)  \right)  ,\label{Def_K}\\
\Phi\left[  a,b\right]  \left(  \mathbf{x}\right)  \overset{\mathrm{def}}%
{=}\prod_{i=a+1}^{b}\phi\left(  x_{i}-x_{i-1}\right)  . \label{Def_Phi}%
\end{gather}
$Z_{n,\lambda,\rho}$ is the usual partition function, i.e., the norming factor
which makes $p_{n,\beta,\rho}$ into a probability density. The main interest
is to prove a central limit theorem for this measure, in the simplest case for
the last marginal measure. It is convenient to consider first the unnormalized
kernel $C_{n}^{\mathrm{SAW}}\left(  x\right)  ,~x\in\mathbb{R}^{d},$ which is
defined to be the last marginal density of $Z_{n,\beta}p_{n,\beta,\rho}\left(
\mathbf{x}\right)  $, i.e.,%
\begin{equation}
C_{n}^{\mathrm{SAW}}\left(  x_{n}\right)  =\int K_{\lambda,\rho}\left[
0,n\right]  \left(  \mathbf{x}\right)  \Phi\left[  0,n\right]  \left(
\mathbf{x}\right)  \prod_{i=1}^{n-1}dx_{i}. \label{Def_C_SAW}%
\end{equation}
By using the lace expansion (as we will show in Section \ref{Subsect_lace}),
the $C_{n}^{\mathrm{SAW}}$ satisfy an equation of the form%
\begin{equation}
C_{n}^{\mathrm{SAW}}=C_{n-1}^{\mathrm{SAW}}\ast\phi+\sum_{k=1}^{n}\Pi_{k}\ast
C_{n-k}^{\mathrm{SAW}}, \label{recursion}%
\end{equation}
where the kernels $\Pi_{k}$ describe the interactions through the weak
self-avoidance. The $\Pi_{k}$ are complicated functions and are hard to
evaluate precisely. However, one crucial property is that the leading order
decay is the same as that of the $C_{k}^{\mathrm{SAW}}.$ It therefore looks
natural to write $\Pi_{k}=\lambda c_{k}^{\mathrm{SAW}}B_{k},$ and one seeks
for conditions on the $B_{k}$ ensuring a CLT for solutions of
(\ref{Basic_ConvolEqu}). We can then apply Theorem \ref{Th_main}, provided we
can check Condition \ref{Cond_Main} on this $\mathbf{B}$ sequence. The theorem
we obtain as a corollary of Theorem \ref{Th_main} is the following:

\begin{theorem}
[Local CLT for WSAW]\label{Th_main_SAW}For $d\geq5$, $\rho\in(0,1]$, and
$\varepsilon>0$ there exists $\lambda_{0}\left(  d,\varepsilon\right)  >0$
such that for all $\lambda\in(0,\lambda_{0}]$ there exist a parameter
$\delta\left(  d,\rho,\lambda\right)  >0$ and a constant $K\left(
d,\varepsilon,\lambda\right)  >0$ such that for all $n\in\mathbb{N}$%
\begin{equation}
\left\vert \frac{C_{n}^{\mathrm{SAW}}\left(  x\right)  }{c_{n}^{\mathrm{SAW}}%
}-\phi_{n\delta}\left(  x\right)  \right\vert \leq K\left[  r_{n}\phi
_{n\delta\left(  1+\varepsilon\right)  }\left(  x\right)  +n^{-d/2}\sum
_{j=1}^{\left\lceil n/2\right\rceil }j\phi_{j\delta\left(  1+\varepsilon
\right)  }\left(  x\right)  \right]  , \label{CLT}%
\end{equation}
with%
\begin{equation}
r_{n}=\left\{
\begin{array}
[c]{cc}%
n^{-1/2} & \mathrm{for\ }d=5,\\
n^{-1}\log n & \mathrm{for\ }d=6,\\
n^{-1} & \mathrm{for\ }d\geq7.
\end{array}
\right.  \label{Def_rn}%
\end{equation}

\end{theorem}

\begin{remark}

\begin{enumerate}
\item[a)] The bound leads to $\left\Vert C_{n}^{\mathrm{SAW}}/c_{n}%
^{\mathrm{SAW}}-\phi_{n\delta}\right\Vert _{1}=O\left(  r_{n}\right)  $.

\item[b)] The theorem does not give a local CLT as at $x=0$ both
$\phi_{n\delta}\left(  0\right)  $ and the bound are of order $n^{-d/2}.$ A
moments reflection however reveals that there cannot be a local CLT as the
starting point keeps to have a noticeable influence on $C_{n}^{\mathrm{SAW}%
}\left(  x\right)  /c_{n}^{\mathrm{SAW}}$ for points $x$ at distance of order
$1$ from the origin. However, our bound proves%
\[
\lim_{r\rightarrow\infty}\limsup_{n\rightarrow\infty}\sup_{x:\left\vert
x\right\vert \geq r}n^{d/2}\left\vert \frac{C_{n}^{\mathrm{SAW}}\left(
x\right)  }{c_{n}^{\mathrm{SAW}}}-\phi_{n\delta}\left(  x\right)  \right\vert
=0.
\]
So the result comes as close as possible to a local CLT.

\item[c)] The summation up to $\left\lceil n/2\right\rceil $ is somewhat
arbitrary, and can be replaced by $\left\lceil \alpha n\right\rceil $ for any
$\alpha\in\left(  0,1\right)  ,$ adapting $K.$ In fact, for $0<\alpha<1$,
there exists a $K(\alpha)$ such that for all $x\in\mathbb{R}^{d}$,
\[
n^{-d/2}\sum_{j=\left\lceil \alpha n\right\rceil }^{n}j\phi_{j\delta\left(
1+\varepsilon\right)  }\left(  x\right)  \leq K\left(  \alpha\right)
r_{n}\phi_{n\delta\left(  1+\varepsilon\right)  }\left(  x\right)  .
\]
We have chosen $\alpha=1/2$ for convenience. The second summand on the right
hand side of (\ref{CLT}) is important as it takes care of the failure of the
local CLT for $x$ near the origin.

\item[d)] The choice of an $\varepsilon>0$ on the right hand side of
(\ref{CLT}) is essentially just for convenience, as it helps to swallow all
kind of polynomial factors in $x$ with which we prefer not to be bothered.
Remark that if the bound (\ref{CLT}) is correct for a positive $\varepsilon
>0,$ it is also true for any larger $\varepsilon,$ with a changed constant
$K.$ It will be convenient to assume that $\varepsilon$ is small, say
$\varepsilon\leq1/100$.
\end{enumerate}
\end{remark}

\subsection{Related literature and structure of the paper}

\label{Literature} Self-avoiding random walks are models for polymer chains of
relevance in statistical physics. Despite their simple definition, a
mathematical rigorous analysis turns out to be a major challange. We refer to
\cite{BaDuGoSl} for a recent survey on this topic. Since the seminal paper by
Brydges and Spencer \cite{BrSp}, the analysis of these models in high
dimensions ($d\geq5$) has been carried out by using the so called \emph{ lace
expansion}. The latter is a diagrammatic type of expansion based on graphs
(which we recall in Section \ref{Subsect_lace}) to deal with combinatorial
objects of relevance in statistical mechanics, e.g. self-avoiding walks,
percolation models, lattice trees. For the interested reader, \cite{Slade}
represents the main reference on this type of expansion. While using the lace
expansion for the analysis of high dimensional WSAW or related models
satisfying equation \eqref{Basic_ConvolEqu}, the procedure is by now standard
and can be roughly summurized via the following three steps.

\begin{enumerate}
\item[1)] Show that the unnormalized densities $C_{n}^{\mathrm{SAW}}$ satisfy
the convolution equation in \eqref{Basic_ConvolEqu}.

\item[2)] Estimate the $B_{k}$ coefficents in \eqref{Basic_ConvolEqu}.

\item[3)] Deduce from the previous steps and equation \eqref{Basic_ConvolEqu}
the growth of the normalized $C_{n}^{\mathrm{SAW}}$ and some detailed Gaussian behavior.
\end{enumerate}

Step 3) is the most involved and technical, especially in \cite{BrSp}. A
successful attempt to simplify this step has been obtained in
\cite{vdHdHSl,vdHSl}, where the authors introduced a new inductive approach.
Both methods in \cite{BrSp, vdHdHSl,vdHSl} heavily rely on spatial Fourier
transforms. In contrast, the method we use does not make use of Fourier
analysis and is based on a fixed point iteration. This novel method is very
different from the previous ones, it was originally developed in the thesis of
Christine Ritzmann \cite{ChRitz,BoRitz}, but never appeared. One of the main
goal of this paper is to present this method with some improvements,
generalizations, and simplifications with respect to \cite{BoRitz,ChRitz}. The
main new feature compared to \cite{BoRitz,ChRitz} is to use a more flexible
and general way to define the operator whose fixed point characterizes the
solution of the convolution equation. Also, in \cite{BoRitz,ChRitz} was
entirely taylored for the application to self-avoiding walks, whereas our main
result on the convolution equations, Theorem \ref{Th_main}, is much more general.

The method gives error bounds in the local CLT that are better than those
obtained with Fourier techniques. The second main novelty of this paper
concerns the application to WSAW in continuous space. In fact, to our
knowledge, all the previous works including \cite{BoRitz,ChRitz} focus on WSAW
on $\mathbb{Z}^{d}$. One of the reason to introduce this variant is that, to
explain our approach based on fixed point iteration, continuous space is
actually more convenient than the lattice. In other words, the emphasis here
is to present an elementary and completely self-contained proof of a sharp CLT
for solutions of (\ref{Basic_ConvolEqu}), together with the perhaps simplest
possible application. No knowledge of earlier versions of lace expansions or
\cite{ChRitz} are assumed.

The rest of this paper is organized as follows. Section \ref{ProofGeneralThm}
is devoted to the proof of the local CLT for general convolution equation,
Theorem \ref{Th_main}. Then, Section \ref{Sect_SAW} focus on the application
to WSAW in continuous space. By performing the three steps sketched above we
show how to derive the local CLT in Theorem \ref{Th_main_SAW} from Theorem
\ref{Th_main}.

\bigskip

\noindent\textbf{Acknowledgement: }We would like to thank the referees for the
careful reading, and the many suggestions which helped to improve the manuscript.

\section{Proof of the local CLT for convolution equations}

\label{ProofGeneralThm} In this section we prove Theorem \ref{Th_main}. The
proof is divided in three main steps which we perform in the following three
sections. First, in Section \ref{cnProp} we analyze the normalizing sequence
$\left\{  c_{n}\right\}  $. In the second step, Section \ref{MainProof}, we
prove Theorem \ref{Th_main} by assuming the technical Lemma \ref{Le_MainEst}
which we prove right after in Section \ref{Subsect_Proof_MainLemma}.

\subsection{On the connectivity constants}

\label{cnProp} A first question we address is about the behavior of the
sequence $\left\{  c_{n}\right\}  $.

\begin{proposition}
\label{Th_Norming}Assume Condition \ref{Cond_Main} and let $\mathbf{c}$ be the
sequence defined by (\ref{Equ_sequ_c}). Then if $\lambda$ is small enough the
following holds:

\begin{enumerate}
\item[a)] There exists a unique $\mu>0$ such that $\alpha\overset
{\mathrm{def}}{=}\lim_{n\rightarrow\infty}\mu^{-n}c_{n}$ exists in $\left(
0,\infty\right)  .$

\item[b)] Writing $a_{n}\overset{\mathrm{def}}{=}\mu^{-n}c_{n},$ one has%
\begin{equation}
\left\vert a_{n+1}-a_{n}\right\vert <L\lambda\overline{\gamma}_{n}%
\overset{\mathrm{def}}{=}L\lambda\sum\nolimits_{j=n}^{\infty}\gamma_{j}.
\label{Sequ3}%
\end{equation}

\item[c)]
\begin{equation}
\mu^{-1}=1-\lambda\sum\nolimits_{k=1}^{\infty}a_{k}b_{k}. \label{Sequ4}%
\end{equation}

\end{enumerate}
\end{proposition}

\begin{remark}

\begin{enumerate}
\item[a)] Plugging the expression (\ref{Sequ4}) into (\ref{Equ_sequ_c}), we
see that $\mathbf{a}=\left\{  a_{n}\right\}  _{n\in\mathbb{N}_{0}}$ satisfies
$a_{0}=1,$ and%
\begin{equation}
a_{n}=a_{n-1}-\lambda a_{n-1}\sum\nolimits_{k=n+1}^{\infty}a_{k}b_{k}%
+\lambda\sum_{k=1}^{n}a_{k}b_{k}\left(  a_{n-k}-a_{n-1}\right)  ,\ n\geq1.
\label{Equ_sequ_a}%
\end{equation}

\item[b)] From (\ref{Sequ3}) we get%
\begin{equation}
\left\vert a_{n}-\alpha\right\vert \leq L\lambda\sum_{k=n}^{\infty}k\gamma
_{k}. \label{alpha}%
\end{equation}

\end{enumerate}
\end{remark}

The idea of the proof is simple: \textit{Assuming }that such a $\mu$ and a
sequence $\left\{  a_{n}\right\}  $ exist, one gets from (\ref{Equ_sequ_c})%
\[
\mu^{n}a_{n}=\mu^{n-1}a_{n-1}+\lambda\mu^{n}\sum_{k=1}^{n}a_{k}b_{k}a_{n-k}.
\]
Letting then $n\rightarrow\infty$, assuming that $\lim_{n\rightarrow\infty
}a_{n}$ exists and is $\neq0,$ one sees that $\mu$ has to be given by
(\ref{Sequ4}) in terms of $\left\{  a_{n}\right\}  .$ Plugging that back, one
arrives at the conclusion, that the $\mathbf{a}$-sequence has to satisfy
(\ref{Equ_sequ_a}). The idea therefore is first to prove by a fixed point
argument that this equation has a nice solution, and then check that%
\[
d_{n}=\left(  1-\lambda\sum\nolimits_{k=1}^{\infty}a_{k}b_{k}\right)
^{-n}a_{n}%
\]
satisfies the equation (\ref{Equ_sequ_c}), and therefore $d_{n}=c_{n},$
finishing the proof.

\begin{proof}
[Proof of Proposition \ref{Th_Norming}]Let $l_{1}\left(  \mathbb{N}\right)  $
be the Banach space of absolutely summable sequences $\mathbf{q}=\left\{
q_{n}\right\}  _{n\in\mathbb{N}},$ and $l_{\gamma}\left(  \mathbb{N}\right)  $
be the set of sequences with $\left\Vert \mathbf{q}\right\Vert _{\gamma
}\overset{\mathrm{def}}{=}\sup_{n}\overline{\gamma}_{n}^{-1}\left\vert
q_{n}\right\vert <\infty.$ $\left(  l_{\gamma}\left(  \mathbb{N}\right)
,\left\Vert \cdot\right\Vert _{\gamma}\right)  $ is a Banach space, too, and
by (\ref{Gamma_Convol4}), $l_{\gamma}\left(  \mathbb{N}\right)  \subset
l_{1}\left(  \mathbb{N}\right)  ,$ and the embedding is continuous. The linear
map $s:l_{1}\left(  \mathbb{N}\right)  \rightarrow l_{\infty}\left(
\mathbb{N}_{0}\right)  $ is defined by $s\left(  \mathbf{q}\right)  _{0}=0,$
and $s\left(  \mathbf{q}\right)  _{n}\overset{\mathrm{def}}{=}\sum_{j=1}%
^{n}q_{j},\ n\geq1.$ \ Evidently, $\left\Vert s\left(  \mathbf{q}\right)
\right\Vert _{\infty}\leq\left\Vert \mathbf{q}\right\Vert _{1}\leq L\left\Vert
\mathbf{q}\right\Vert _{\gamma}$. We also define the affine mapping
$S:l_{1}\left(  \mathbb{N}\right)  \rightarrow l_{\infty}\left(
\mathbb{N}_{0}\right)  $ by $S\left(  \mathbf{q}\right)  \overset
{\mathrm{def}}{=}\mathbf{1+}s\left(  \mathbf{q}\right)  $, where $\mathbf{1}$
is the sequence identical to $1$. We define two mappings $\psi_{1},\psi_{2}$
from $l_{1}\left(  \mathbb{N}\right)  $ to the set of sequences with index set
$\mathbb{N}$. We set%
\begin{gather*}
\psi_{1}\left(  \mathbf{q}\right)  _{n}\overset{\mathrm{def}}{=}S\left(
\mathbf{q}\right)  _{n-1}\sum_{k=n+1}^{\infty}b_{k}S\left(  \mathbf{q}\right)
_{k},\\
\psi_{2}\left(  \mathbf{q}\right)  _{n}\overset{\mathrm{def}}{=}\sum_{k=2}%
^{n}S\left(  \mathbf{q}\right)  _{k}b_{k}\left[  s\left(  \mathbf{q}\right)
_{n-k}-s\left(  \mathbf{q}\right)  _{n-1}\right]
\end{gather*}
for $n\geq1$. Finally we set $\psi\overset{\mathrm{def}}{=}-\lambda\psi
_{1}+\lambda\psi_{2}.$ Remark first that%
\[
\psi\left(  \mathbf{0}\right)  _{n}=\lambda\psi_{1}\left(  \mathbf{0}\right)
_{n}=\lambda\sum_{k=n+1}^{\infty}b_{k},
\]
where $\mathbf{0}$ is the sequence identical to $0.$ We conclude that
$\left\Vert \psi\left(  \mathbf{0}\right)  \right\Vert _{\gamma}\leq
L\lambda,$ by (\ref{bound_gamma}).%
\begin{align*}
\left\vert \psi_{1}\left(  \mathbf{q}\right)  _{n}-\psi_{1}\left(
\mathbf{p}\right)  _{n}\right\vert  &  \leq\left\Vert s\left(  \mathbf{q}%
\right)  -s\left(  \mathbf{p}\right)  \right\Vert _{\infty}\left[
\sum_{k=n+1}^{\infty}\left\vert b_{k}S\left(  \mathbf{q}\right)
_{k}\right\vert +\left\vert S\left(  \mathbf{p}\right)  _{n-1}\right\vert
\sum_{k=n+1}^{\infty}\left\vert b_{k}\right\vert \right] \\
&  \leq L\left\Vert \mathbf{q}-\mathbf{p}\right\Vert _{\gamma}\left[
2+L\left\Vert \mathbf{q}\right\Vert _{\gamma}+L\left\Vert \mathbf{p}%
\right\Vert _{\gamma}\right]  \sum_{k=n+1}^{\infty}\gamma_{k}%
\end{align*}%
\[
\left\Vert \psi_{1}\left(  \mathbf{q}\right)  -\psi_{1}\left(  \mathbf{p}%
\right)  \right\Vert _{\gamma}\leq L\left\Vert \mathbf{q}-\mathbf{p}%
\right\Vert _{\gamma}\left(  1+\left\Vert \mathbf{q}\right\Vert _{\gamma
}+\left\Vert \mathbf{p}\right\Vert _{\gamma}\right)  .
\]
Similarly, for $n\geq2,$ by resummation%
\begin{align}
\psi_{2}\left(  \mathbf{q}\right)  _{n}-\psi_{2}\left(  \mathbf{p}\right)
_{n}  &  =\sum_{j=1}^{n-1}q_{j}\sum_{k=n-j+1}^{n}\left(  S\left(
\mathbf{p}\right)  _{k}-S\left(  \mathbf{q}\right)  _{k}\right)
b_{k}\label{Sequ5}\\
&  +\sum_{j=1}^{n-1}\left(  p_{j}-q_{j}\right)  \sum_{k=n-j+1}^{n}S\left(
\mathbf{p}\right)  _{k}b_{k}.\nonumber
\end{align}
In the first summand, we estimate $\left\vert S\left(  \mathbf{q}\right)
_{k}-S\left(  \mathbf{p}\right)  _{k}\right\vert $ by $L\left\Vert
\mathbf{q}-\mathbf{p}\right\Vert _{\gamma},$ so we get for this part an
estimate%
\begin{equation}
\leq L\left\Vert \mathbf{q}\right\Vert _{\gamma}\left\Vert \mathbf{q}%
-\mathbf{p}\right\Vert _{\gamma}\sum_{j=1}^{n-1}\sum_{t=j}^{\infty}\gamma
_{t}\sum_{k=n-j+1}^{n}\gamma_{k}. \label{Est1}%
\end{equation}
Further,
\begin{align}
\sum_{j=1}^{n-1}\sum_{t=j}^{\infty}\gamma_{t}\sum_{k=n-j+1}^{n}\gamma_{k}  &
\leq\sum_{j=1}^{n-1}\sum_{t=j}^{\infty}\sum_{k=n-j+1}^{n}\chi_{t+k}\left(
t\right)  \gamma_{t+k}\nonumber\\
&  \leq\sum_{s=n+1}^{\infty}\gamma_{s}\sum_{t=1}^{s-1}N\left(  s,t\right)
\chi_{s}\left(  t\right)  , \label{Est2}%
\end{align}
where we have used (\ref{bound_gamma}), and where $N\left(  s,t\right)  $ is
the number of indices $j$ satisfying $1\leq j\leq n-1,\ t\geq j,\ n-j+1\leq
s-t\leq n,$ so that $N\left(  s,t\right)  \leq t\wedge(s-t),$ and using
(\ref{Sequ1}), from \eqref{Est1} and \eqref{Est2}, we get for the first
summand of (\ref{Sequ5}) an estimate $\leq L\left\Vert \mathbf{\mathbf{q}%
}\right\Vert _{\gamma}\left\Vert \mathbf{q}-\mathbf{p}\right\Vert _{\gamma
}\overline{\gamma}_{n}.$ In a similar way, we get for the second summand an
estimate $\leq L\left(  1+\left\Vert \mathbf{p}\right\Vert _{\gamma}\right)
\left\Vert \mathbf{q}-\mathbf{p}\right\Vert _{\gamma}\overline{\gamma}_{n}$
and therefore%
\[
\left\Vert \psi_{2}\left(  \mathbf{q}\right)  -\psi_{2}\left(  \mathbf{p}%
\right)  \right\Vert _{\gamma}\leq L\left\Vert \mathbf{q}-\mathbf{p}%
\right\Vert _{\gamma}\left(  1+\left\Vert \mathbf{q}\right\Vert _{\gamma
}+\left\Vert \mathbf{p}\right\Vert _{\gamma}\right)  ,
\]
leading to%
\[
\left\Vert \psi\left(  \mathbf{q}\right)  -\psi\left(  \mathbf{p}\right)
\right\Vert _{\gamma}\leq L\lambda\left\Vert \mathbf{q}-\mathbf{p}\right\Vert
_{\gamma}\left(  1+\left\Vert \mathbf{q}\right\Vert _{\gamma}+\left\Vert
\mathbf{p}\right\Vert _{\gamma}\right)  .
\]
From that and $\psi\left(  \mathbf{0}\right)  \in l_{\gamma}\left(
\mathbb{N}\right)  ,$ it follows that $\psi$ maps $l_{\gamma}\left(
\mathbb{N}\right)  $ continuously into itself, and furthermore, if $\lambda$
is small enough, the iterates $\psi^{n}\left(  \mathbf{0}\right)  $ form a
Cauchy sequence, and therefore converge in $l_{\gamma}\left(  \mathbb{N}%
\right)  $ to an element $\mathbf{\xi}$ with $\left\Vert \mathbf{\xi
}\right\Vert _{\gamma}\leq L\lambda$ which is a fixed point of $\psi.$

If we write%
\[
\mathbf{\eta}\overset{\mathrm{def}}{=}S\left(  \mathbf{\xi}\right)
,\ \varpi\overset{\mathrm{def}}{=}\left(  1-\lambda\sum\nolimits_{k=1}%
^{\infty}\eta_{k}b_{k}\right)  ^{-1},
\]
then it is evident, using the fact that $\mathbf{\xi}$ is a fixed point of
$\psi,$ that the sequence $\mathbf{\eta}$ satisfies (\ref{Equ_sequ_a}),
implying that the sequence $\left\{  \eta_{n}\varpi^{n}\right\}  $ satisfies
(\ref{Equ_sequ_c}), and therefore it \textit{is }this sequence. So it follows
that $\varpi=\mu,$ and $\mu^{-n}c_{n}$ satisfies the properties listed in a)-c).
\end{proof}

\subsection{Proof of Theorem \ref{Th_main}\label{MainProof}}

Before giving the proof, let us first start with a few observations.

As $B_{m}\in\mathcal{C}_{\ast}\left(  \mathbb{R}^{d}\right)  ,$ the
\textquotedblleft covariance\textquotedblright\ matrix satisfies%
\begin{equation}
\int x^{T}xB_{m}\left(  x\right)  dx=\overline{b}_{m}I_{d}, \label{Def_bm_bar}%
\end{equation}
for some $\overline{b}_{m}\in\mathbb{R}$ (possibly negative), $I_{d}$ being
the $d\times d$ unit matrix. Evidently, $\left\vert \overline{b}%
_{m}\right\vert \leq\gamma_{m}^{\left(  1\right)  },$ and by Condition
\ref{Cond_Main} (\ref{Gamma_Convol4}), the following number is well defined
(for small enough $\lambda$):%

\begin{equation}
\delta\overset{\mathrm{def}}{=}\frac{\mu^{-1}+\lambda\sum_{m=1}^{\infty}%
a_{m}\overline{b}_{m}}{\mu^{-1}+\lambda\sum_{m=1}^{\infty}ma_{m}b_{m}},
\label{Def_delta}%
\end{equation}
where $\mu$ and $b_{m}$ are given by \eqref{Sequ4} and \eqref{bm},
respectively. In particular, by choosing $\lambda>0$ small enough, we can
achieve that%
\begin{equation}
\left\vert 1-\delta\right\vert \leq L\lambda,\ \left\vert 1-\mu\right\vert
\leq L\lambda. \label{delta_bound}%
\end{equation}
and also%
\[
1/2\leq a_{n}\leq3/2,\ \forall n,
\]
which we assume henceforward.

The idea of the proof of Theorem \ref{Th_main} is to consider an appropriate
Banach space of sequences of functions with a norm that encodes the error we
expect in the local CLT. We then prove that $\left\{  \mu^{-n}C_{n}-\mu
^{-n}c_{n}\phi_{n\delta}\right\}  _{n\in\mathbb{N}}$ is an element of this
Banach space by proving that it appears as a limit of a Cauchy sequence. This
implies the desired result.

Let us start by describing the Banach space we need. Let $\mathbf{f}=\left\{
f_{n}\right\}  $ be a sequence of functions in $\mathcal{C}_{\ast}^{+}\left(
\mathbb{R}^{d}\right)  $ which satisfy $\lim_{n\rightarrow\infty}\sup_{x}%
f_{n}\left(  x\right)  =0.$ For any sequence $\mathbf{g}=\left\{
g_{n}\right\}  $,\ $g_{n}\in\mathcal{C}_{\ast}\left(  \mathbb{R}^{d}\right)  $
define%
\[
\left\Vert \mathbf{g}\right\Vert _{\mathbf{f}}\overset{\mathrm{def}}{=}%
\sup_{n}\sup_{x\in\mathbb{R}^{d}}\frac{\left\vert g_{n}\left(  x\right)
\right\vert }{f_{n}\left(  x\right)  },
\]
and write $\mathcal{B}_{\mathbf{f}}\overset{\mathrm{def}}{=}\left\{
\mathbf{g}:\left\Vert \mathbf{g}\right\Vert _{\mathbf{f}}<\infty\right\}  $
which equipped with $\left\Vert \mathbf{\cdot}\right\Vert _{\mathbf{f}}$ is a
Banach space.

For our purposes, we consider the Banach space $\left(  \mathcal{B}%
_{\mathbf{f}},\left\Vert \mathbf{\cdot}\right\Vert _{\mathbf{f}}\right)  $
with $\mathbf{f}=\left\{  f_{n}\right\}  $ defined by
\begin{equation}
f_{n}\overset{\mathrm{def}}{=}\sum_{s=1}^{\left[  n/2\right]  }s\psi_{s}%
\ast\Gamma_{n-s}+\overline{\zeta}_{n}\psi_{n}, \label{Def_fn}%
\end{equation}
where $\psi_{n}\overset{\mathrm{def}}{=}\phi_{n\delta\left(  1+\varepsilon
\right)  }$. (As remarked before, the choice of $\varepsilon>0$ is only of
minor relevance, but it influences the notion of \textquotedblleft small
enough $\lambda$\textquotedblright). Note that the sequence $\left\{
f_{n}\right\}  $ is the same as the sequence of error terms in the right hand
side of \eqref{mainbound}.

Next, let $\mathbf{C}$ be the solution of (\ref{Basic_ConvolEqu}) and put
$A_{n}\overset{\mathrm{def}}{=}C_{n}\mu^{-n}$. This sequence satisfies
$A_{0}=\delta_{0}$ and%
\begin{equation}
A_{n}=\mu^{-1}A_{n-1}\ast\phi+\lambda\sum_{k=1}^{n}a_{k}B_{k}\ast A_{n-k},
\label{Def_SequA}%
\end{equation}
where $a_{n}=\int A_{n}\left(  x\right)  dx,$ and $A_{n}/a_{n}=C_{n}/c_{n}.$

In particular, note that the statement of Theorem \ref{Th_main} is equivalent
(given Proposition \ref{Th_Norming}) to bound $\left\vert A_{n}\left(
x\right)  -a_{n}\phi_{n\delta}\left(  x\right)  \right\vert $ in the same way,
and this is what we will do.

We define the following operator $\Psi$ on sequences of functions
$\mathbf{G}=\left\{  G_{n}\right\}  _{n\geq0}$, $G_{n}\in\mathcal{C}_{\ast
}\left(  \mathbb{R}^{d}\right)  ,\ $ $\Psi\left(  \mathbf{G}\right)
_{0}\overset{\mathrm{def}}{=}G_{0},$ and for $n\geq1$%
\[
\Psi\left(  \mathbf{G}\right)  _{n}\overset{\mathrm{def}}{=}a_{n}\phi
_{n\delta}\ast G_{0}-\sum_{j=1}^{n}G_{n-j}\ast\Delta_{j,j},
\]
with%
\begin{equation}
\Delta_{k,j}\overset{\mathrm{def}}{=}a_{j}\phi_{k\delta}-\mu^{-1}a_{j-1}%
\phi_{\left(  k-1\right)  \delta+1}-\lambda\sum_{m=1}^{j}a_{m}a_{j-m}B_{m}%
\ast\phi_{\left(  k-m\right)  \delta} \label{Def_Delta}%
\end{equation}
for $k\geq j.$ A resummation gives%
\[
\Psi\left(  \mathbf{G}\right)  _{n}=G_{n}-\sum_{j=1}^{n}a_{n-j}\phi_{\left(
n-j\right)  \delta}\ast\left[  G_{j}-\mu^{-1}\phi\ast G_{j-1}-\lambda
\sum_{m=1}^{j}a_{m}B_{m}\ast G_{j-m}\right]  .
\]

A crucial observation is that if $\mathbf{A}$ satisfies $A_{0}=\delta_{0}$ and
(\ref{Def_SequA}), then $\Psi\left(  \mathbf{A}\right)  =\mathbf{A}$, and vice
versa: If $A_{0}=\delta_{0},$ and $\mathbf{A}$ satisfies the fixed point
equation, then (\ref{Def_SequA}) follows by induction on $n.$

The main technical estimates are summarized in the following lemma which will
be proved in the next section

\begin{lemma}
\label{Le_MainEst}

\begin{enumerate}
\item[a)]
\begin{equation}
\sum_{j=1}^{n}\left\vert \Delta_{n,j}\right\vert \leq L\lambda f_{n},
\label{ML1}%
\end{equation}

\item[b)]
\begin{equation}
\left\vert \Delta_{n,n}\right\vert \leq L\lambda\kappa_{n}, \label{ML2}%
\end{equation}
where%
\begin{equation}
\kappa_{n}\overset{\mathrm{def}}{=}\sum_{s=0}^{\left[  n/2\right]  }\psi
_{s}\ast\Gamma_{n-s}+n^{-1}\bar{\zeta}_{n}\psi_{n}, \label{Chi_Def}%
\end{equation}

\item[c)]
\begin{equation}
\sum_{j=1}^{n}\kappa_{j}\ast f_{n-j}\leq Lf_{n}. \label{ML3}%
\end{equation}

\end{enumerate}
\end{lemma}

We proceed with the proof of Theorem \ref{Th_main}, assuming this lemma. Note
that on the one hand, if $\mathbf{E}$ is the sequence $\left\{  a_{n}%
\phi_{n\delta}\right\}  $ then $\Psi\left(  \mathbf{E}\right)  _{n}=E_{n}%
-\sum_{j=1}^{n}a_{n-j}\Delta_{n,j}.$ By Lemma \ref{Le_MainEst} a), we get that
$\Psi\left(  \mathbf{E}\right)  -\mathbf{E}\in\mathcal{B}_{\mathbf{f}}$ with
$\left\Vert \Psi\left(  \mathbf{E}\right)  -\mathbf{E}\right\Vert
_{\mathbf{f}}\leq L\lambda$. ($\mathbf{E}$ itself is of course not in
$\mathcal{B}_{\mathbf{f}}$).

On the other hand, if $\mathbf{G}\in\mathcal{B}_{\mathbf{f}},$ with $G_{0}=0,$
then for $n\geq1,$%
\[
\left\vert \Psi\left(  \mathbf{G}\right)  _{n}\left(  x\right)  \right\vert
\leq\left\Vert \mathbf{G}\right\Vert _{\mathbf{f}}\sum_{j=1}^{n}\left\vert
f_{n-j}\left(  x\right)  \Delta_{j,j}\left(  x\right)  \right\vert .
\]
By applying Lemma \ref{Le_MainEst} b) and c), we obtain that%
\[
\left\Vert \Psi\left(  \mathbf{G}\right)  \right\Vert _{\mathbf{f}}\leq
L\lambda\left\Vert \mathbf{G}\right\Vert _{\mathbf{f}}.
\]

Thus, since $\left(  \Psi\left(  \mathbf{E}\right)  -\mathbf{E}\right)
_{0}=0$, we conclude that for small enough $\lambda>0$, $\left\{  \Psi
^{n}\left(  \mathbf{E}\right)  -\mathbf{E}\right\}  $ is a Cauchy sequence in
$\mathcal{B}_{\mathbf{f}}$, and therefore converges, say to $\mathbf{Y}%
\in\mathcal{B}_{\mathbf{f}}$ which satisfies $\left\Vert \mathbf{Y}\right\Vert
_{\mathbf{f}}\leq L\lambda.$ Then%
\begin{align*}
\mathbf{Y}+\mathbf{E}-\Psi\left(  \mathbf{Y}+\mathbf{E}\right)   &  =\left[
\mathbf{Y}+\mathbf{E}-\Psi^{n}\left(  \mathbf{E}\right)  \right] \\
&  +\left[  \Psi^{n}\left(  \mathbf{E}\right)  -\Psi^{n+1}\left(
\mathbf{E}\right)  \right]  +\left[  \Psi^{n+1}\left(  \mathbf{E}\right)
-\Psi\left(  \mathbf{Y}+\mathbf{E}\right)  \right]  ,
\end{align*}
and all three expressions in square brackets on the right hand side converge
to $0$ in $\mathcal{B}_{\mathbf{f}}$. Therefore, $\mathbf{Y}+\mathbf{E}$ is a
fixed point of $\Psi$, which we know has to be $\mathbf{A}$. Therefore
$\left\Vert \mathbf{A}-\mathbf{E}\right\Vert _{\mathbf{f}}\leq L\lambda.$ So,
we have proved the theorem.

\subsection{Proof of Lemma \ref{Le_MainEst}\label{Subsect_Proof_MainLemma}}

We first recall some properties of the semigroup $\left\{  \phi_{t}\right\}
$. Of course, $\phi_{t}\left(  x\right)  =t^{-d/2}\phi\left(  x/\sqrt
{t}\right)  .$ We often write $\dot{\phi}_{t}$ for the derivative in $t,$ and
we write $\partial_{i}\phi_{t}$ for the partial derivatives in $x_{i},$ and
$\partial_{ij}^{2}\phi_{t}$ for the second partial derivatives, etc. We also
write $\Delta\phi_{t}\overset{\mathrm{def}}{=}\sum\nolimits_{i=1}^{d}%
\partial_{ii}^{2}\phi_{t},$ as usual. The heat equation gives $\dot{\phi}%
_{t}=\frac{1}{2}\Delta\phi_{t}.$ The partial derivatives in $x$ of $\phi$ are
of the form $p\phi$ for a polynomial $p$ in $x$ whose exact form is of no
concern for us. Here are some elementary properties we will use:

\begin{itemize}
\item If $t\leq s\leq2t$ then%
\begin{equation}
\phi_{t}\leq2^{d/2}\phi_{s}. \label{Phi1}%
\end{equation}

\item If $p$ is any polynomial in $x,$ then for any $\varepsilon>0,$ there
exists $C_{\varepsilon,p}>0$ such that%
\begin{equation}
\left\vert p\left(  x\right)  \right\vert \phi\left(  x\right)  \leq
C_{\varepsilon,p}\phi_{1+\varepsilon}\left(  x\right)  \label{Phi2}%
\end{equation}
implying%
\begin{equation}
\left\vert p\left(  x/\sqrt{t}\right)  \right\vert \phi_{t}\left(  x\right)
\leq C_{\varepsilon,p}\phi_{t\left(  1+\varepsilon\right)  }\left(  x\right)
. \label{Phi3}%
\end{equation}
From this, we see that for $\mathbf{k}=\left(  k_{1},\ldots,k_{d}\right)
\in\mathbb{N}_{0}^{d}$ with $\left\vert \mathbf{k}\right\vert =k_{1}%
+\cdots+k_{d}$%
\begin{equation}
\left\vert \frac{\partial^{\left\vert \mathbf{k}\right\vert }\phi_{t}\left(
x\right)  }{\partial x_{1}^{k_{1}}\cdots\partial x_{d}^{k_{d}}}\right\vert
\leq C_{\varepsilon,k}t^{-\left\vert \mathbf{k}\right\vert /2}\phi_{t\left(
1+\varepsilon\right)  }\left(  x\right)  , \label{Der1}%
\end{equation}
and for $k\in\mathbb{N}$%
\begin{equation}
\left\vert \frac{\partial^{k}\phi_{t}\left(  x\right)  }{\partial t^{k}%
}\right\vert \leq C_{\varepsilon,k}t^{-k}\phi_{t\left(  1+\varepsilon\right)
}\left(  x\right)  . \label{Der2}%
\end{equation}

\end{itemize}

Below, we use the convention $\sum_{m=a}^{b}=0$ if $b<a.$

\subsubsection{Proof of (\ref{ML1})}

Recall that $a_{n}=\mu^{-n}c_{n}$. Using (\ref{Sequ4}) and (\ref{Equ_sequ_a}),
we can rewrite $\Delta_{k,j}$ as%
\begin{align*}
\Delta_{k,j}  &  =\mu^{-1}a_{j-1}\left(  \phi_{k\delta}-\phi_{\left(
k-1\right)  \delta+1}\right)  -\lambda\sum_{m=1}^{j}a_{m}a_{j-m}\left(
B_{m}\ast\phi_{\left(  k-m\right)  \delta}-b_{m}\phi_{k\delta}\right) \\
&  =\Delta_{k,j}^{(1)}+\Delta_{k,j}^{(2)},
\end{align*}
where%
\[
\Delta_{k,j}^{(2)}\overset{\mathrm{def}}{=}-\lambda\sum_{m=\left[  k/2\right]
+1}^{j}a_{m}a_{j-m}\left(  B_{m}\ast\phi_{\left(  k-m\right)  \delta}%
-b_{m}\phi_{k\delta}\right)  .
\]
Note that $\Delta^{(1)}$ and $\Delta^{(2)}$ deal with small and large $m$,
respectively. As $\left\{  a_{m}\right\}  $ is bounded, we can estimate, using
$\left\vert b_{m}\right\vert \leq\gamma_{m},\ \left\vert B_{m}\right\vert
\leq\Gamma_{m},$%
\begin{equation}
\left\vert \Delta_{k,j}^{(2)}\right\vert \leq L\lambda\left[  \sum
_{s=k-j}^{\left[  k/2\right]  }\phi_{s\delta}\ast\Gamma_{k-s}+\phi_{k\delta
}\sum_{m=\left[  k/2\right]  +1}^{j}\gamma_{m}\right]  , \label{Est_Delta2}%
\end{equation}
where in the first summand on the rhs, we substituted $k-m=s.$ From that we
see that (\ref{ML1}) (and also (\ref{ML2})) hold for $\Delta^{(2)}$ instead of
$\Delta$ and so it remains to check the inequalities for $\Delta^{(1)}:$%

\begin{align*}
\Delta_{k,j}^{(1)}  &  =\alpha\left[  \mu^{-1}\left(  \phi_{k\delta}%
-\phi_{\left(  k-1\right)  \delta+1}\right)  -\lambda\sum_{m=1}^{j\wedge
\left[  k/2\right]  }a_{m}\left(  B_{m}\ast\phi_{\left(  k-m\right)  \delta
}-b_{m}\phi_{k\delta}\right)  \right] \\
&  +\mu^{-1}\left(  a_{j-1}-\alpha\right)  \left(  \phi_{k\delta}%
-\phi_{\left(  k-1\right)  \delta+1}\right) \\
&  -\lambda\sum_{m=1}^{j\wedge\left[  k/2\right]  }a_{m}\left(  a_{j-m}%
-\alpha\right)  \left(  B_{m}\ast\phi_{\left(  k-m\right)  \delta}-b_{m}%
\phi_{k\delta}\right)  ,\\
&  =X_{k,j}^{(1)}+X_{k,j}^{(2)}-X_{k,j}^{(3)},\ \mathrm{say.}%
\end{align*}
To estimate $X^{(1)}$, we use a Taylor approximation $\phi_{t}\left(
x\right)  $ in the $x$-variable up to fourth order. Remark that in the
expansion below, the odd contributions vanish due to the assumed symmetry of
the $B_{m}$ function, and in the second Taylor term, we replace $\frac{1}%
{2}\Delta\phi_{t}$ by $\dot{\phi}_{t}.$ $b_{m}$ and $\overline{b}_{m}$ are
defined by (\ref{bm}) and (\ref{Def_bm_bar}).%
\begin{align*}
\left(  B_{m}\ast\phi_{\left(  k-m\right)  \delta}\right)  \left(  x\right)
&  =b_{m}\phi_{\left(  k-m\right)  \delta}\left(  x\right)  +\overline{b}%
_{m}\dot{\phi}_{\left(  k-m\right)  \delta}\left(  x\right) \\
&  +\frac{1}{24}E_{\theta}\left(  \int\phi_{\left(  k-m\right)  \delta
}^{\left(  4\right)  }\left(  x-\theta y\right)  \left[  y^{4}\right]
B_{m}\left(  y\right)  dy\right)  ,
\end{align*}
where $E_{\theta}$ refers to an expectation under the probability measure with
density $4\left(  1-\theta\right)  ^{3}$ on $\left[  0,1\right]  .$
$\phi^{\left(  4\right)  }\left(  z\right)  \left[  y^{4}\right]  $ is the
fourth derivative of $\phi$ at $z$ in the direction $y.$ The third summand, we
estimate by (\ref{Gamma_Convol3}) and (\ref{Der1}), using $m<k/2:$%
\begin{align*}
&  \leq Lk^{-2}E_{\theta}\int\phi_{\left(  k-m\right)  \delta\left(
1+\varepsilon\right)  }\left(  x-\theta y\right)  \left\vert y\right\vert
^{4}\Gamma_{m}\left(  y\right)  dy\\
&  =Lk^{-2}E_{\theta}\theta^{-d}\int\phi_{\left(  k-m\right)  \delta\left(
1+\varepsilon\right)  /\theta^{2}}\left(  \frac{x}{\theta}-y\right)
\left\vert y\right\vert ^{4}\Gamma_{m}\left(  y\right)  dy\\
&  \leq Lk^{-2}\gamma_{m}^{\left(  2\right)  }E_{\theta}\theta^{-d}%
\phi_{\left(  k-m\right)  \delta\left(  1+\varepsilon\right)  /\theta^{2}%
+m}\left(  \frac{x}{\theta}\right) \\
&  =Lk^{-2}\gamma_{m}^{\left(  2\right)  }E_{\theta}\phi_{\left(  k-m\right)
\delta\left(  1+\varepsilon\right)  +m\theta^{2}}\left(  x\right) \\
&  \leq Lk^{-2}\gamma_{m}^{\left(  2\right)  }\psi_{k}\left(  x\right)
\end{align*}
as $\theta^{2}\leq\delta\left(  1+\varepsilon\right)  $ if $\lambda$ is small
enough (by (\ref{delta_bound})). Furthermore%
\begin{align*}
\overline{b}_{m}\dot{\phi}_{\left(  k-m\right)  \delta}  &  =\overline{b}%
_{m}\dot{\phi}_{k\delta}+O\left(  \gamma_{m}^{\left(  1\right)  }mk^{-2}%
\psi_{k}\right)  ,\\
b_{m}\phi_{\left(  k-m\right)  \delta}  &  =b_{m}\phi_{k\delta}-b_{m}%
m\delta\dot{\phi}_{k\delta}+O\left(  \gamma_{m}m^{2}k^{-2}\psi_{k}\right)  ,\\
\phi_{\left(  k-1\right)  \delta+1}  &  =\phi_{k\delta}+\left(  1-\delta
\right)  \dot{\phi}_{k\delta}+O\left(  k^{-2}\lambda^{2}\psi_{k}\right)  .
\end{align*}
So we get%
\[
X_{k,j}^{(1)}=\left[  \mu^{-1}\left(  1-\delta\right)  -\lambda\sum
_{m=1}^{j\wedge\left(  k/2\right)  }a_{m}\left(  \overline{b}_{m}-b_{m}%
m\delta\right)  \right]  \dot{\phi}_{k\delta}+O\left(  \lambda k^{-2}%
\zeta_{j\wedge\left[  k/2\right]  }^{\left(  1\right)  }\psi_{k}\right)  .
\]
The choice of $\delta$ was made such that the expression in square brackets is
$0$ if we extend the sum to $\infty.$ Therefore, the expression in square
brackets is in absolute value%
\[
\leq L\lambda\sum_{m\geq j\wedge\left(  k/2\right)  }\left(  \left\vert
\overline{b}_{m}\right\vert +m\left\vert b_{m}\right\vert \right)  \leq
L\lambda\sum_{m\geq j\wedge\left(  k/2\right)  }\left(  \gamma_{m}^{\left(
1\right)  }+m\gamma_{m}\right)  \leq L\lambda\zeta_{j\wedge\left[  k/2\right]
}^{\left(  2\right)  },
\]
and as $\left\vert \dot{\phi}_{k\delta}\right\vert \leq Lk^{-1}\psi_{k}$, we
get%
\begin{equation}
\left\vert X_{k,j}^{(1)}\right\vert \leq L\lambda\left\{  k^{-2}\zeta
_{j\wedge\left[  k/2\right]  }^{\left(  1\right)  }+k^{-1}\zeta_{j\wedge
\left[  k/2\right]  }^{\left(  2\right)  }\right\}  \psi_{k}. \label{Est_X1}%
\end{equation}

For $X^{(2)}$, we simply use $\phi_{\left(  k-1\right)  \delta+1}%
=\phi_{k\delta}+O\left(  \lambda k^{-1}\psi_{k}\right)  ,$ and Proposition
\ref{Th_Norming} c) to get%
\begin{equation}
\left\vert X_{k,j}^{(2)}\right\vert \leq L\lambda k^{-1}\zeta_{j}^{\left(
2\right)  }\psi_{k}, \label{Est_X2}%
\end{equation}
and in a similar fashion, we get%
\begin{equation}
\left\vert X_{k,j}^{(3)}\right\vert =L\lambda k^{-1}\zeta_{j\wedge\left[
k/2\right]  }^{\left(  2\right)  }\psi_{k}. \label{Est_X3}%
\end{equation}
Using these estimates for $X^{(1)},X^{(2)},X^{(3)},$ we get%
\begin{align*}
\sum_{j=1}^{n}\left\vert \Delta_{n,j}^{(1)}\right\vert  &  \leq L\lambda
\left\{  n^{-2}\sum_{j=1}^{n}\zeta_{j\wedge\left[  n/2\right]  }^{\left(
1\right)  }+n^{-1}\sum_{j=1}^{n}\zeta_{j\wedge\left[  n/2\right]  }^{\left(
2\right)  }\right\}  \psi_{n}\\
&  \leq L\lambda\left\{  n^{-2}\sum_{j=1}^{n}\zeta_{j}^{\left(  1\right)
}+n^{-1}\sum_{j=1}^{n}\zeta_{j}^{\left(  2\right)  }\right\}  \psi
_{n}=L\lambda\overline{\zeta}_{n}\psi_{n},
\end{align*}
i.e., the estimate (\ref{ML1}) for $\Delta^{(1)}$.

\subsubsection{Proof of (\ref{ML2})}%

\begin{equation}
\left\vert \Delta_{j,j}^{(1)}\right\vert \leq L\lambda\left\{  j^{-2}\zeta
_{j}^{\left(  1\right)  }+j^{-1}\zeta_{j}^{\left(  2\right)  }\right\}
\psi_{j}\leq\frac{L\lambda}{j}\overline{\zeta}_{j}\psi_{j}.
\label{Inequ_Delta0}%
\end{equation}
The first inequality is evident by (\ref{Est_X1})-(\ref{Est_X3}). To see the
second one, remark first that $\zeta_{j}^{\left(  2\right)  }$ is decreasing
in $j,$ and therefore $\zeta_{j}^{\left(  2\right)  }\leq\overline{\zeta}_{j}$
follows. It remains to prove $j^{-1}\zeta_{j}^{\left(  1\right)  }\leq
L\overline{\zeta}_{j}$ which is the same as to prove%
\begin{equation}
1+\sum_{i=0}^{2}\sum_{m=1}^{j}m^{2-i}\gamma_{m}^{\left(  i\right)  }\leq
L+Lj^{-1}\sum_{i=0}^{2}\sum_{m=1}^{j}\left(  j-m+1\right)  m^{2-i}\gamma
_{m}^{\left(  i\right)  }. \label{Inequ_Delta1}%
\end{equation}
If we restrict both sides to summations over $m\leq2j/3,$ the inequality is
evident. On the other hand, using the assumed monotonicity of the $\gamma
_{n}^{\left(  i\right)  }$ sequences, we have%
\begin{align*}
\sum_{m=2j/3}^{j}m^{2-i}\gamma_{m}^{\left(  i\right)  }  &  \leq j^{2-i}%
\sum_{m=2j/3}^{j}\gamma_{m}^{\left(  i\right)  }\leq j^{2-i}\sum
_{m=j/3}^{2j/3}\gamma_{m}^{\left(  i\right)  }\\
&  \leq27j^{-1}\sum_{m=j/3}^{2j/3}\left(  j-m+1\right)  m^{2-i}\gamma
_{m}^{\left(  i\right)  }.
\end{align*}

As we had $\left\vert \Delta_{j,j}^{(2)}\right\vert \leq\frac{L\lambda}%
{j}\overline{\zeta}_{j}\psi_{j}$ already by (\ref{Est_Delta2}), the proof is complete.

\subsubsection{Proof of (\ref{ML3})}

Recall (\ref{Def_fn}), and write $f_{n}=f_{n}^{\left(  1\right)  }%
+f_{n}^{\left(  2\right)  }$ where $f_{n}^{(1)}$ is the first of the two
summands, and $f_{n}^{(2)}$ the second. We similarly split $\kappa_{n}%
=\kappa_{n}^{(1)}+\kappa_{n}^{\left(  2\right)  }$.

Using \eqref{Gamma_Convol1}, estimate
\begin{align*}
\sum_{j=1}^{n}\kappa_{j}^{(1)}\ast f_{n-j}^{(1)}  &  =\sum_{j=1}^{n-1}%
\sum_{s=0}^{\left[  j/2\right]  }\sum_{t=1}^{\left[  \left(  n-j\right)
/2\right]  }t\left(  \psi_{s}\ast\psi_{t}\right)  \ast\left(  \Gamma_{j-s}%
\ast\Gamma_{n-j-t}\right) \\
&  \leq L\sum_{j=1}^{n-1}\sum_{s=0}^{\left[  j/2\right]  }\sum_{t=1}^{\left[
\left(  n-j\right)  /2\right]  }t\chi_{n-s-t}\left(  j-s\right)  \left(
\psi_{s+t}\ast\Gamma_{n-s-t}\right) \\
&  \leq L\sum_{r=1}^{\left[  n/2\right]  }\rho\left(  r\right)  \left(
\psi_{r}\ast\Gamma_{n-r}\right)
\end{align*}
with%
\[
\rho\left(  r\right)  \overset{\mathrm{def}}{=}\sum_{j=1}^{n-1}\sum
_{s=0\vee\left(  r-\left[  \left(  n-j\right)  /2\right]  \right)  }^{\left[
j/2\right]  \wedge\left(  r-1\right)  }\left(  r-s\right)  \chi_{n-r}\left(
j-s\right)  \leq r\sum_{k=1}^{n-r-1}\alpha_{n,r}\left(  k\right)  \chi
_{n-r}\left(  k\right)  ,
\]
with%
\[
\alpha_{n,r}\left(  k\right)  \overset{\mathrm{def}}{=}\#\left\{  \left(
j,s\right)  :1\leq j\leq n-1,\ j-s=k,\ 0\vee\left(  r-\left[  \left(
n-j\right)  /2\right]  \right)  \leq s\leq\left[  j/2\right]  \wedge\left(
r-1\right)  \right\}  .
\]
It is elementary to check that $\alpha_{n,r}\left(  k\right)  \leq\min\left(
k,2\left(  n-r-k\right)  \right)  ,$ which implies by (\ref{Sequ1})
$\rho\left(  r\right)  \leq2K_{1}r$, so we get%
\begin{equation}
\sum_{j=1}^{n}\kappa_{j}^{(1)}\ast f_{n-j}^{(1)}\leq Lf_{n}. \label{MR1}%
\end{equation}

We next estimate
\begin{equation}
\sum_{j=1}^{n}\kappa_{j}^{(1)}\ast f_{n-j}^{(2)}=\sum_{j=1}^{n}\sum
_{s=0}^{\left[  j/2\right]  }\overline{\zeta}_{n-j}\Gamma_{j-s}\ast
\psi_{n-j+s}. \label{MR3}%
\end{equation}
For the summands with $j-s\leq\left[  n/2\right]  $ we have by
(\ref{Gamma_Convol3}) $\Gamma_{j-s}\ast\psi_{n-j+s}\leq L\gamma_{j-s}\ast
\psi_{n}$ and by (\ref{Bound_zetabar}),\ as $n-j\geq n/4,$ we have
$\overline{\zeta}_{n-j}\leq L\overline{\zeta}_{n}$. So we get for this part of
the sum on the rhs%
\[
\leq L\overline{\zeta}_{n}\psi_{n}\sum_{j=1}^{n}\sum_{s:s\leq\left[
j/2\right]  ,\ j-s\leq\left[  n/2\right]  }\gamma_{j-s}\leq L\overline{\zeta
}_{n}\psi_{n}.
\]
For the summands on the rhs of (\ref{MR3}) with $j-s>\left[  n/2\right]  ,$ we
get, by substituting $k$ for $n-j+s$, that it is $\leq\sum_{k=1}^{\left[
n/2\right]  }\left[  \sum_{s\leq\left(  n-k\right)  /2}\right]  \overline
{\zeta}_{k-s}\psi_{k}\Gamma_{n-k}\leq\sum_{k=1}^{\left[  n/2\right]  }%
k\psi_{k}\Gamma_{n-k}$, so that we have proved%
\begin{equation}
\sum_{j=1}^{n}\kappa_{j}^{(1)}\ast f_{n-j}^{(2)}\leq Lf_{n}. \label{MR2}%
\end{equation}

We next prove
\begin{equation}
\sum_{j=1}^{n}\kappa_{j}^{(2)}\ast f_{n-j}^{(1)}\leq Lf_{n}. \label{MR4}%
\end{equation}

\[
\sum_{j=1}^{n}\kappa_{j}^{(2)}\ast f_{n-j}^{(1)}=\sum_{j=1}^{n-1}\frac
{\bar{\zeta}_{j}}{j}\sum_{s=1}^{\left[  \left(  n-j\right)  /2\right]
}s\left[  \psi_{j+s}\ast\Gamma_{n-j-s}\right]  .
\]
We split $Q\overset{\mathrm{def}}{=}\left\{  \left(  j,s\right)  :1\leq j\leq
n-1,\ 1\leq s\leq\left[  \left(  n-j\right)  /2\right]  \right\}  $ into the
part $Q_{1}$ with $j+s\leq n/2,$ the part $Q_{2}$ with $n/2<j+s\leq3n/4$, and
the part $Q_{3}$ with $j+s>3n/4.$ On $Q_{2}\cup Q_{3}$ we again use
(\ref{Gamma_Convol3}) and estimate $\psi_{j+s}\ast\Gamma_{n-j-s}\leq
L\gamma_{n-j-s}\psi_{n}$. On $Q_{3}$, we must have $j\geq n/4,$ and therefore%
\[
\sum_{Q_{3}}\frac{\bar{\zeta}_{j}}{j}s\gamma_{n-j-s}\leq L\frac{\bar{\zeta
}_{n}}{n}\sum_{Q_{3}}s\gamma_{n-j-s}\leq L\bar{\zeta}_{n}.
\]%
\begin{align*}
\sum_{Q_{2}}\frac{\bar{\zeta}_{j}}{j}s\gamma_{n-j-s}  &  \leq L\gamma_{n}%
\sum_{Q_{2}}\frac{\bar{\zeta}_{j}}{j}\leq Ln\gamma_{n}\sum_{j=1}^{\infty}%
\frac{\bar{\zeta}_{j}}{j}\\
&  \leq Ln\gamma_{n}\leq L\bar{\zeta}_{n},
\end{align*}
the last inequality by (\ref{Bound_Gamman_by_zetabar}). Finally,%
\begin{align*}
\sum_{Q_{1}}\frac{\bar{\zeta}_{j}}{j}s\left[  \psi_{j+s}\ast\Gamma
_{n-j-s}\right]   &  =\sum_{k=1}^{\left[  n/2\right]  }\psi_{k}\ast
\Gamma_{n-k}\sum_{Q_{1}\cap\left\{  \left(  j,s\right)  :j+s=k\right\}  }%
\frac{\bar{\zeta}_{j}}{j}s\\
&  \leq L\sum_{k=1}^{\left[  n/2\right]  }k\left(  \psi_{k}\ast\Gamma
_{n-k}\right)  .
\end{align*}
Therefore, we have proved (\ref{MR4}).

Finally, it remains to investigate
\[
\sum_{j=1}^{n}\kappa_{j}^{(2)}\ast f_{n-j}^{(2)}=\psi_{n}\sum_{j=1}^{n}%
\frac{\overline{\zeta}_{j}}{j}\overline{\zeta}_{n-j}.
\]
The summation over $j\leq n/2$ is $\leq L\overline{\zeta}_{n}\sum_{j}%
\overline{\zeta}_{j}/j\leq L\overline{\zeta}_{n}$ by (\ref{Bound_zetabar}),
and the summation over $j>n/2$ is $\leq\left(  \overline{\zeta}_{n}/n\right)
\sum_{j\leq n}\overline{\zeta}_{j}\leq\overline{\zeta}_{n}\sum_{j}\left(
\overline{\zeta}_{j}/j\right)  \leq L\overline{\zeta}_{n}.$ Therefore%
\begin{equation}
\sum_{j=1}^{n}\kappa_{j}^{(2)}\ast f_{n-j}^{(2)}\leq Lf_{n}. \label{MR5}%
\end{equation}
Combining (\ref{MR1}),\ (\ref{MR2}), (\ref{MR4}), and (\ref{MR5}) proves the claim.

\section{Application to weakly self-avoiding walks: Proof of
Theorem\ref{Th_main_SAW} \label{Sect_SAW}}

We choose an $\varepsilon$ with $0<\varepsilon\leq1/100$ which will be fixed
through the rest of this section.

We derive Theorem \ref{Th_main_SAW} by applying the main Theorem \ref{Th_main}
with%
\begin{equation}
\Gamma_{n}\overset{\mathrm{def}}{=}Kn^{-d/2}\sum_{k=1}^{n}k^{1-d/2}\phi
_{2k/5}, \label{Def_Gamma_SAW}%
\end{equation}
with%
\begin{equation}
K\overset{\mathrm{def}}{=}8\mathrm{e}^{5/4}\left(  1+\frac{3}{2}\left(
1+\frac{1}{100}\right)  ^{d/2}\right)  \label{K_choice}%
\end{equation}

Let us first show that this $\Gamma_{n}$ satisfies B1-B4 in Condition
\ref{Cond_Main}:

\begin{lemma}
If $d\geq5,$ then the sequence $\left\{  \Gamma_{n}\right\}  $ defined in
\eqref{Def_Gamma_SAW} satisfies B1-B4 from Condition \ref{Cond_Main}.
\end{lemma}

\begin{proof}
B2 and B4 are readily checked.

B1:
\begin{align*}
\Gamma_{n}\ast\Gamma_{m}  &  =K^{2}\left(  nm\right)  ^{-d/2}\sum_{k\leq
n}\sum_{l\leq m}\left(  kl\right)  ^{1-d/2}\phi_{2\left(  k+l\right)  /5}\\
&  =K^{2}\left(  \frac{n+m}{nm}\right)  ^{d/2}\left(  n+m\right)  ^{-d/2}%
\sum_{t=2}^{n+m}\left(  \sum_{k=1}^{t-1}\left(  k\left(  t-k\right)  \right)
^{1-d/2}\right)  \phi_{2t/5}\\
&  \leq C\left(  d\right)  \left(  \frac{n+m}{nm}\right)  ^{d/2}\Gamma_{n+m},
\end{align*}
for some constant $C\left(  d\right)  >0$ depending only on $d,$ which proves
B1. Note that the last inequality holds only when $d\geq5$.

B3: \ We use the fact that $\left\vert y\right\vert ^{2k}\phi_{j}\leq
Lj^{k}\phi_{3j/2}$ for $j\in\mathbb{N}$ and $k=0,1,2.$ Therefore, we have for
$m\leq t$%
\begin{align*}
\int\phi_{t}\left(  \cdot-y\right)  \left\vert y\right\vert ^{2k}\Gamma
_{m}\left(  y\right)  dy  &  \leq C\left(  d\right)  m^{-d/2}\sum_{j=1}%
^{m}j^{1-d/2+k}\phi_{t+3j/5}\\
&  \leq C\left(  d\right)  \phi_{t+m}m^{-d/2}\sum_{j=1}^{m}j^{1-d/2+k}\leq
L\gamma_{m}^{\left(  k\right)  }\phi_{t+m}.
\end{align*}

\end{proof}

We keep our convention of the last section concerning the constant $L$.
However, as we have chosen $\varepsilon$ fixed, and a concrete $\Gamma$ which
specifies $K_{1}-K_{6},$ depending only on the dimension $d\geq5$, $L$ now
depends only on the dimension $d$.

With this choice of $\Gamma$, we have $\bar{\zeta}_{n}=O\left(  r_{n}\right)
$ , where $r_{n}$ is defined in (\ref{Def_rn}), and therefore the bound in
Theorem \ref{Th_main} is%
\begin{align}
&  L\left[  \sum_{s=1}^{\left[  n/2\right]  }s\left(  \phi_{s\delta\left(
1+\varepsilon\right)  }\ast\left(  n-s\right)  ^{-d/2}\sum_{k=1}%
^{n-s}k^{1-d/2}\phi_{2k/5}\right)  \left(  x\right)  +r_{n}\phi_{n\delta
\left(  1+\varepsilon\right)  }\left(  x\right)  \right]  \nonumber\\
&  \leq L\left[  n^{-d/2}\sum_{s=1}^{\left[  n/2\right]  }s\left(
\phi_{s\delta\left(  1+\varepsilon\right)  }\ast\sum_{k=1}^{n-s}k^{1-d/2}%
\phi_{2k/5}\right)  \left(  x\right)  +r_{n}\phi_{n\delta\left(
1+\varepsilon\right)  }\left(  x\right)  \right]  \label{Resum}\\
&  \leq L\left[  n^{-d/2}\sum_{s=1}^{\left[  n/2\right]  }s\phi_{s\delta
\left(  1+\varepsilon\right)  }\left(  x\right)  +r_{n}\phi_{n\delta\left(
1+\varepsilon\right)  }\left(  x\right)  \right]  ,\nonumber
\end{align}
the last inequality provided%
\begin{equation}
\delta\left(  1+\varepsilon\right)  \geq4/5,\label{Restriction_epsilon}%
\end{equation}
which is achieved by choosing $\lambda$ small enough. To see the second
inequality in (\ref{Resum}), we sum $sk^{1-d/2}\phi_{s\delta\left(
1+\varepsilon\right)  +2k/5}$ over $s,k$ satisfying $s\delta\left(
1+\varepsilon\right)  +2k/5\in(s^{\prime}-1,s^{\prime}]\delta\left(
1+\varepsilon\right)  $, estimate $\phi_{s\delta\left(  1+\varepsilon\right)
+2k/5}$ by $L\phi_{s^{\prime}\delta\left(  1+\varepsilon\right)  }$, and
finally sum over $s^{\prime}.$ This leads to%
\[
L\sum_{s^{\prime}}s^{\prime}\phi_{s^{\prime}\delta\left(  1+\varepsilon
\right)  }\left(  x\right)
\]
but the summation extends beyond $\left[  n/2\right]  .$ However, the sum over
$s^{\prime}>\left[  n/2\right]  $ can be estimated by $Ln^{d/2}r_{n}%
\phi_{n\delta\left(  1+\varepsilon\right)  }\left(  x\right)  $ provided all
the $s^{\prime}$ are $\leq n\delta\left(  1+\varepsilon\right)  $ which is
guranteed by (\ref{Restriction_epsilon}).

In order to prove Theorem \ref{Th_main_SAW} we have to show that the
connectivity function in \eqref{Def_C_SAW} satisfies the recursion in
\eqref{recursion}. This is done in Section \ref{Subsect_lace}. Finally, we
have to show that the $B_{n}$'s, defined through $\Pi_{n}=\lambda
c_{n}^{\mathrm{SAW}}B_{n}$, are bounded from above by the $\Gamma_{n}$
sequence in \eqref{Def_Gamma_SAW}. This is the content of Section
\ref{Subsect_Lacebound}.

There is nothing mysterious in our choice of $\left\{  \Gamma_{n}\right\}  $:
Simply \textit{assume }that a (near) local CLT is correct. Then estimating the
$B_{n}$ for WSAW from the lace expansion, immediately leads to an estimate
$\left\vert B_{n}\right\vert \leq\Gamma_{n}$. On the other hand, $\left\vert
B_{n}\right\vert \leq\Gamma_{n}$ implies a (near) local CLT. There is
sufficient \textquotedblleft contraction\textquotedblright\ in this circle to
make it work.

\subsection{Definition of the Lace Functions and recursion for
WSAW\label{Subsect_lace}}

This section contains standard material on the lace expansion adapted to the
model in continuous space.

Given an interval $I=[a,b]\subset\mathbb{Z}$ of integers with $0\leq a\leq b
$, we refer to a pair $\{s,t\}$ ($s<t$) of elements of $I$ as an \emph{edge}.
To abbreviate the notation, we write $st$ for $\{s,t\}$. A set of edges is
called a \emph{graph}. A graph $\Gamma$ on $[a,b]$ is said to be
\emph{connected} if both $a$ and $b$ are endpoints of edges in $\Gamma$ and
if, in addition, for any $c\in\lbrack a,b]$ there is an edge $st\in\Gamma$
such that $s<c<t$. Note that this is \textit{not }in agreement with the usual
notion of connectedness in graph theory. The set of all graphs on $[a,b]$ is
denoted by $\mathcal{B}[a,b]$, and the subset consisting of all connected
graphs is denoted by $\mathcal{G}[a,b]$. A \emph{lace} is a minimally
connected graph, that is, a connected graph for which the removal of any edge
would result in a disconnected graph. The set of laces on $[a,b]$ is denoted
by $\mathcal{L}[a,b]$, and the set of laces on $[a,b]$ consisting of exactly
$N$ edges is denoted by $\mathcal{L}^{(N)}[a,b]$.

A lace $\ell=\left\{  s_{1}t_{1},\ldots,s_{N}t_{N}\right\}  $ on $\left[
0,n\right]  $, with $s_{1}=0,\ t_{N}=n,$ satisfies $s_{i}<t_{i-1}%
,\ i=2,\ldots,N,$ and $t_{i}\leq s_{i+2},\ i=1,\ldots,N-2.$ We can describe
the lace by the interdistances $m_{1},\ldots,m_{2N-1}$ between the points
$s_{i},t_{i}$ ordered increasingly, $s_{1}=0<s_{2}<t_{1}\leq s_{3}<t_{2}%
\cdots,$ i.e. $m_{1}=s_{2},\ m_{2}=t_{1}-s_{2},$ etc. Then of course
$\sum_{i=1}^{2N-1}m_{i}=n.$ We switch freely between the $s_{i}$-$t_{i}%
$-representation of the lace and the representation by the $m_{i}$, without
special notice. The restrictions on the $m_{i}$ are $m_{i}>0$ for $i$ even and
$m_{i}\geq0$ for $i$ odd, with the additional restriction at the boundary
$m_{1}>0$ and $m_{2N-1}>0.$ (For $N=2,$ all the $m_{i}$ are positive). It is
customary the visualize the laces as graphs by identifying the vertices
connected by a bond. Below the example of a lace with $N=4.$
%\texttt{Graph 1 hier einf\"{u}gen.}

\medskip

\begin{figure}[h]
\centering
\includegraphics[width=0.7\linewidth,clip]{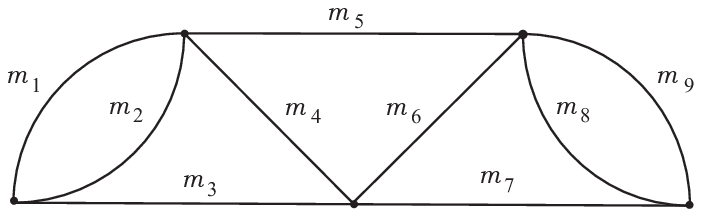}
%\caption{\small Illustration of the retraction $\eta _k$}
%\protect\label{figure12}
\end{figure}

\medskip

The \textquotedblleft basic\textquotedblright\ $N$-lace is the graph%
\begin{equation}
\ell_{N}^{0}\overset{\mathrm{def}}{=}\left\{  \left(  0,2\right)  ,\left(
1,4\right)  ,\left(  3,6\right)  ,\ldots,\left(  2N-5,2N-2\right)  ,\left(
2N-3,2N-1\right)  \right\}  \label{Def_lN0}%
\end{equation}
on $\left\{  0,\ldots,2N-1\right\}  .$ We will write $b_{i}=\left(
\underline{i},\overline{i}\right)  $ for the $i$-th bond in this graph, i.e.
$\underline{1}=0,$ and $\underline{i}=2i-3$ for $i=1,\ldots,N,$ $\overline
{i}=2i$ for $i\leq N-1$, and $\overline{N}=2N-1.$ Conversely, for
$i=0,\ldots,2N-1$, we write $\beta\left(  i\right)  \in\left\{  1,\ldots
,N\right\}  $ for the unique element with $i\in\left\{  \underline
{\beta\left(  i\right)  },\overline{\beta\left(  i\right)  }\right\}  ,$ i.e.
$\beta\left(  0\right)  =1,\ \beta\left(  1\right)  =2,$ etc. With this
notation, we have for a lace in $\mathcal{L}^{(N)}[a,b]$%
\[
s_{i}=\sum_{j=1}^{\underline{i}}m_{j},\ t_{i}=\sum_{j=1}^{\overline{i}}m_{j}.
\]

If $G=\left\{  G_{t}\right\}  _{t>0}$ is any family of functions in
$\mathcal{C}_{\ast}^{+}\left(  \mathbb{R}^{d}\right)  $, augmented by
$G_{0}=\delta_{0}$, and $\ell\in\mathcal{L}^{\left(  N\right)  }\left[
0,n\right]  ,$ we write with $x_{0}=0,\ x_{2N-1}=x$,%
\begin{equation}
\Xi_{\ell}\left(  G,\rho\right)  \left(  x\right)  \overset{\mathrm{def}}%
{=}\int dx_{1}\cdots dx_{2N-2}\prod\limits_{i=1}^{2N-1}G_{m_{i}}\left(
x_{i}-x_{i-1}\right)  \prod\limits_{i=1}^{n}\mathbb{I}_{\rho}\left(
x_{\overline{i}}-x_{\underline{i}}\right)  . \label{Def_Xi}%
\end{equation}
For the moment, we need $G$ only for integer $m$, but the more general
situation is needed below.

Given a connected graph $\Gamma$ on $\left[  a,b\right]  $, the following
prescription associates to $\Gamma$ a unique lace $\ell_{\Gamma}$. The lace
consists of edges $s_{1}t_{1},s_{2}t_{2},\dots$, with $t_{1}$, $s_{1}$,
$t_{2}$, $s_{2}$, $\dots$ determined (in that order) by%

\begin{align*}
t_{1}  &  = \max\{t:\, at\in\Gamma\}, & s_{1}  &  = a,\\
t_{i+1}  &  = \max\{t:\, \exists s< t_{i}\,\text{ such that }st\in\Gamma\}, &
s_{i+1}  &  = \min\{s:\, st_{i+1}\in\Gamma\}.
\end{align*}

Given a lace $\mathcal{\ell}$, the set of all edges $st\notin\mathcal{\ell}$
such that $\mathcal{\ell}_{\mathcal{\ell}\cup\{st\}}=\mathcal{\ell}$ is
denoted by $\mathcal{C}(\mathcal{\ell})$. Edges in $\mathcal{C}(\mathcal{\ell
})$ are said to be \emph{compatible} with $\mathcal{\ell}$. With this
formalism, we can expand the product in (\ref{Def_K}), obtaining%

\begin{align}
\label{Kab}K_{\lambda,\rho}\left[  a,b\right]  \left(  \mathbf{x}\right)
=\sum_{\Gamma\in\mathcal{B}[a,b]} \prod_{st\in\Gamma}\left(  -\lambda
U_{st}^{\rho}\left(  \mathbf{x}\right)  \right)  .
\end{align}

We also define an analogous quantity, in which the sum over graphs is
restricted to connected graphs, namely,
\begin{align}
\label{JDef}J[a,b]\left(  \mathbf{x}\right)  \overset{\mathrm{def}}{=}%
\sum_{\Gamma\in\mathcal{G}[a,b]} \prod_{st\in\Gamma}\left(  -\lambda
U_{st}^{\rho}\left(  \mathbf{x}\right)  \right)  .
\end{align}

Recalling (\ref{Def_Phi}), this allows us to define the \emph{lace functions},
which are the key quantities in the lace expansion:
\begin{equation}
\Pi_{n}(x_{n})\overset{\mathrm{def}}{=}\int J[0,n]\left(  \mathbf{x}\right)
\Phi\left[  0,n\right]  \left(  \mathbf{x}\right)  \prod_{i=1}^{n-1}dx_{i}
\label{Pimdef}%
\end{equation}
for any $n\geq1$ and $x_{n}\in\mathbb{R}^{d}$. The identity (\ref{recursion})
is shown in the following lemma.

\begin{lemma}
[Convolution equation for WSAW]\label{lem-recursion} For $n\geq1$,
\[
C_{n}^{\mathrm{SAW}}=C_{n-1}^{\mathrm{SAW}}\ast\phi+\sum_{k=1}^{n}\Pi_{k}\ast
C_{n-k}^{\mathrm{SAW}}.
\]

\end{lemma}

\begin{proof}
It suffices to show that for each path $\mathbf{x}$ we have (suppressing
$\mathbf{x}$ in the formulas):
\begin{align}
\label{recurs2}K[0,n] = K[1,n] + \sum_{m=1}^{n} J[0,m] \, K[m,n].
\end{align}
Then \eqref{recursion} is obtained after insertion of \eqref{recurs2} into
\eqref{Def_C_SAW} followed by factorization of the integral over $\mathbf{x}$.
To prove \eqref{recurs2}, we note from \eqref{Kab} that the contribution to
$K[0,n]$ from all graphs $\Gamma$ for which $0$ is not in an edge is exactly
$K[1,n]$. To resum the contribution from the remaining graphs, we proceed as
follows. When $\Gamma$ does contain an edge ending at $0$, we let $m[\Gamma]$
denote the largest value of $m$ such that the set of edges in $\Gamma$ with at
least one end in the interval $[0,m]$ forms a connected graph on $[0,m]$. Then
resummation over graphs on $[m,n]$ gives%

\begin{align}
K[0,n] = K[1,n] + \sum_{m=1}^{n} \sum_{\Gamma\in\mathcal{G}[0,m]} \prod
_{st\in\Gamma}(-\lambda U_{st})\, K[m,n].
\end{align}
With \eqref{JDef} this proves \eqref{recurs2}.
\end{proof}

We next rewrite \eqref{Pimdef} in a form that can be used to obtain good
bounds on $\Pi_{n}(x)$. First, splitting the sum over $\Gamma\in
\mathcal{G}\left[  a,b\right]  $ according to the number of bonds in
$\ell_{\Gamma},$ we get%
\[
J[a,b]=\sum\nolimits_{N\geq1}J_{N}\left[  a,b\right]  ,
\]%
\begin{align}
&  J_{N}\left[  a,b\right]  \overset{\mathrm{def}}{=}\sum_{\ell\in
\mathcal{L}^{\left(  N\right)  }\left[  a,b\right]  }\sum_{\Gamma:\ell
_{\Gamma}=\ell}\prod_{st\in\ell}\left(  -\lambda U_{st}\right)  \prod
_{s^{\prime}t^{\prime}\in\Gamma\backslash\ell}\left(  -\lambda U_{s^{\prime
}t^{\prime}}\right) \nonumber\\
&  =\left(  -\lambda\right)  ^{N}\sum_{\ell\in\mathcal{L}^{\left(  N\right)
}\left[  a,b\right]  }\prod_{st\in\ell}U_{st}\prod_{s^{\prime}t^{\prime}%
\in\mathcal{C}\left(  \ell\right)  }\left(  1-\lambda U_{s^{\prime}t^{\prime}%
}\right) \label{Def_JN}\\
&  =\left(  -\lambda\right)  ^{N}J^{\left(  N\right)  }\left[  a,b\right]
,\ \mathrm{say}.\nonumber
\end{align}
Implementing into \eqref{Pimdef}, we get a splitting%
\[
\Pi_{n}=\sum\nolimits_{N\geq1}\left(  -\lambda\right)  ^{N}\Pi_{n}^{\left(
N\right)  },
\]
where $\Pi_{n}^{\left(  N\right)  }$ is obtained by replacing $J\left[
0,n\right]  $ in (\ref{Pimdef}) by $J^{\left(  N\right)  }\left[  0,n\right]
.$ Note that the sum over $N$ is restricted to $N<n.$

An important point is that we obtain an upper bound for $\Pi_{n}^{\left(
N\right)  }$ by dropping in (\ref{Def_JN}) the factors $\left(  1-\lambda
U_{s^{\prime}t^{\prime}}\right)  $ for all $s^{\prime}t^{\prime}$ which cross
an endpoint of any $st$ bond of the lace $\ell$. This gives the upper bound%
\begin{equation}
\Pi_{n}^{\left(  N\right)  }\left(  x\right)  \leq\sum_{\ell\in\mathcal{L}%
^{\left(  N\right)  }\left[  0,n\right]  }\Xi_{\ell}\left(  C,\rho\right)
\left(  x\right)  \label{Monotonicity}%
\end{equation}
for $N\geq2,$ where $C=\left\{  C_{n}\right\}  .$ For $N=1$, there is the
slight modification from \textquotedblleft restoring\textquotedblright\ the
$0n$ bond: $\Pi_{n}^{\left(  1\right)  }\left(  x\right)  =\Xi_{0n}\left(
C,\rho\right)  \left(  x\right)  /\left(  1-\lambda\right)  .$

\subsection{Bounds on the lace function\label{Subsect_Lacebound}}

We need below a slight generalization of the notion in (\ref{Def_Xi}). Given
$G_{t}$, defined for real $t>0,$ we define for an additional sequence
$\mathbf{t}=\left(  t_{1},\ldots,t_{2N-1}\right)  ,\ \Xi_{\ell}\left(
G,\rho,\mathbf{t}\right)  \left(  x\right)  $ by replacing $m_{i}$ on the
right hand side of (\ref{Def_Xi}) by $m_{i}+t_{i}.$ Also, given an arbitrary
sequence $\mathbf{r}=\left(  r_{1},\ldots,r_{2N-1}\right)  $ of elements in
$\mathbb{N}_{0},$ we write%
\[
\xi_{n}^{\left(  N\right)  }\left(  G,\rho,\mathbf{t},\mathbf{r}\right)
\left(  x\right)  \overset{\mathrm{def}}{=}\sum_{\mathbf{m}\in\mathcal{L}%
^{\left(  N\right)  }\left[  0,n\right]  ,\ m_{i}\geq r_{i}}\Xi_{\ell}\left(
G,\rho,\mathbf{t}\right)  \left(  x\right)  .
\]
Of course, finally we are interested only in the case where the $r_{i}$ are
the \textquotedblleft natural\textquotedblright\ ones from the restriction of
the laces, i.e. $r_{1}=r_{2}=1,\ r_{3}=0$ (if $N\geq3$) etc. We write
$\mathbf{r}^{\left(  0\right)  }$ for this starting sequence. If $\mathbf{t}$
is the sequence of $0$'s, and $\mathbf{r}=\mathbf{r}^{\left(  0\right)  },$ we
drop these arguments in the notation. We will need the more general ones in an
induction argument.

We first state a simple lemma regarding normal densities.

\begin{lemma}
If $u,v,s,t>0,\ x,y\in\mathbb{R}^{d},$ then%
\begin{equation}
\int\phi_{u}\left(  z\right)  \phi_{\nu}\left(  x-z\right)  \phi_{s}\left(
z\right)  \phi_{t}\left(  y-z\right)  dz\leq L\left[  \frac{u+v}{uv}\right]
^{d/4}\left[  \frac{s+t}{st}\right]  ^{d/4}\phi_{u+v}\left(  x\right)
\phi_{s+t}\left(  y\right)  . \label{Convol}%
\end{equation}

\end{lemma}

\begin{proof}
By Cauchy-Schwarz the left hand side is%
\[
\leq\sqrt{\int\phi_{u}^{2}\left(  z\right)  \phi_{\nu}^{2}\left(  x-z\right)
dz}\sqrt{\int\phi_{s}^{2}\left(  z\right)  \phi_{t}^{2}\left(  y-z\right)
dz},
\]
which equals the rhs of (\ref{Convol}) by an elementary computation.
\end{proof}

Let us fix some more notation. We saw that an $N$-lace is nothing but a
sequence $\mathbf{m}=\left(  m_{1},\ldots,m_{2N-1}\right)  $ with $\sum
_{i}m_{i}=n,$ and satisfying some restrictions, like $m_{1}\geq1,\ m_{2}%
\geq1,\ m_{3}\geq0$ . We write $\mathbf{r}^{\left(  0\right)  }=\left(
1,1,0,1,0,\ldots\right)  $ for this sequence of restrictions. For an arbitrary
sequence $\mathbf{r}\in\mathbb{N}_{0}^{2N-1}$ with $\sum_{i}r_{i}\leq n,$ we
write $\mathcal{L}_{\mathbf{r}}^{\left(  N\right)  }\left[  0,n\right]  $ for
the set of $\mathbf{m}$ satisfying $m_{i}\geq r_{i},\ \forall i,$ and
$\sum_{i}m_{i}=n.$ The $r_{i}$ need not satisfy $r_{i}\geq r_{i}^{\left(
0\right)  }.$

\begin{lemma}
\label{Le_reduce}For $\nu>0$, $\mathbf{m}\in\mathbb{N}_{0}^{2N-1},\ t_{i}%
\geq0,$ $\mathbf{x}=\left(  x_{1},\ldots,x_{N-1}\right)  \in\left(
\mathbb{R}^{d}\right)  ^{N-1},$ let%
\[
\Phi_{N,\mathbf{m},\mathbf{t}}^{\left(  \nu\right)  }\left(  \mathbf{x}%
\right)  \overset{\mathrm{def}}{=}\prod_{i=1}^{2N-1}\phi_{\nu m_{i}+t_{i}%
}\left(  x_{\beta\left(  i\right)  -1}-x_{\beta\left(  i-1\right)  -1}\right)
,
\]
with $x_{0}=0.$ If for any $i$ either $r_{i}\geq1$ or $t_{i}\geq c,$ then, for
$d\geq5$ and $\ N\geq3$,%
\[
\sum_{\mathbf{m}\in\mathcal{L}_{\mathbf{r}}^{\left(  N\right)  }\left[
0,n\right]  }\int dx_{1}\Phi_{N,\mathbf{m},\mathbf{t}}^{\left(  \nu\right)
}\left(  \mathbf{x}\right)  \leq L\left(  c\right)  \sum_{\mathbf{m}^{\prime
}\in\mathcal{L}_{\mathbf{r}^{\prime}}^{\left(  N-1\right)  }\left[
0,n\right]  }\Phi_{N-1,\mathbf{m}^{\prime},\mathbf{t}^{\prime}}^{\left(
\nu\right)  }\left(  x_{2},\ldots,x_{N-1}\right)  ,
\]
where $\mathbf{r}^{\prime}\overset{\mathrm{def}}{=}\left(  r_{3},r_{1}%
+r_{4},r_{2}+r_{5},r_{6},\ldots,r_{2N-1}\right)  ,\ \mathbf{t}^{\prime
}\overset{\mathrm{def}}{=}\left(  t_{3},t_{1}+t_{4},t_{2}+t_{5},t_{6}%
,\ldots,t_{2N-1}\right)  $ which both have $2N-3$ components.
\end{lemma}

\begin{proof}
The part of $\Phi_{N,\mathbf{m},\mathbf{t}}^{\left(  \nu\right)  }\left(
\mathbf{x}\right)  $ which contains $x_{1}$ is%
\[
\phi_{m_{1}\nu+t_{1}}\left(  x_{1}\right)  \phi_{m_{2}\nu+t_{2}}\left(
x_{1}\right)  \phi_{m_{4}\nu+t_{4}}\left(  x_{2}-x_{1}\right)  \phi_{m_{5}%
\nu+t_{5}}\left(  x_{3}-x_{1}\right)  .
\]
In case $N=3,$ we have $x_{3}=x_{2}.$ Using the previous lemma for the
integration over $x_{1}$, and summing over $m_{1},m_{2},m_{4},m_{5},$ keeping
$m_{1}+m_{4}=m_{2}^{\prime},\ m_{2}+m_{5}=m_{3}^{\prime}$ fixed, we get for
the $x_{1}$-integration and this restricted summation of the above expression
a bound%
\[
\leq L\left(  c\right)  \phi_{m_{2}^{\prime}\nu+t_{1}+t_{4}}\left(
x_{2}\right)  \phi_{m_{3}^{\prime}\nu+t_{2}+t_{5}}\left(  x_{3}\right)  .
\]
We write $\mathbf{m}^{\prime}\in\mathbb{N}_{0}^{2N-3}$ with $m_{1}^{\prime
}=m_{3}, m_{2}^{\prime}=m_{1}+m_{4}, m_{3}^{\prime}=m_{2}+m_{5}$, and
$m_{i}^{\prime}=m_{i+2}$ otherwise. The restrictions on the $m_{i}^{\prime}$
are evidently given by $m_{i}^{\prime}\geq r_{i}^{\prime}.$ Summing over
$\mathbf{m}^{\prime}$ gives the desired bound.
\end{proof}

Here is the illustration of the \textquotedblleft collapsing mechanism":
%\texttt{Graph 2 hier einf\"{u}gen}

\medskip

\begin{figure}[h]
\centering
\includegraphics[width=1.0\linewidth,clip]{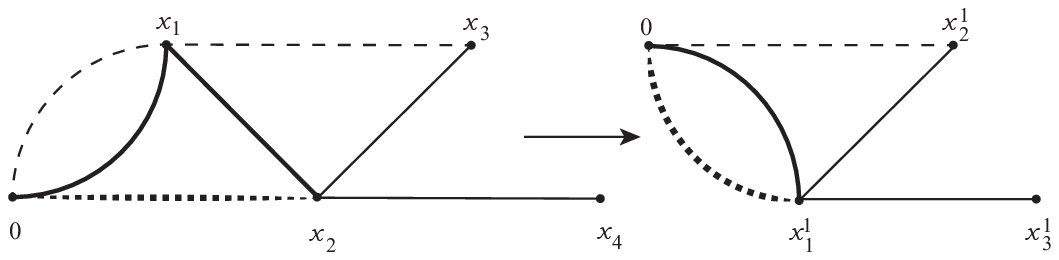}
%\caption{\small Illustration of the retraction $\eta _k$}
%\protect\label{figure2}
\end{figure}

\medskip

\begin{lemma}
\label{Le_main_induction}Assume $d\geq5$. If for some $\nu\in\left[  \frac
{19}{20},\frac{21}{20}\right]  $ and $m\in\mathbb{N},$ $m\geq3$, one has%
\begin{equation}
G_{n}\left(  x\right)  \leq\phi_{n\nu}\left(  x\right)  , \label{Assumption}%
\end{equation}
for all $n<m,$ then for $N\geq2,\ 0<\rho\leq1$, we have with $L=L\left(
d\right)  $, not depending on $m,N$%
\[
\xi_{m}^{\left(  N\right)  }\left(  G,\rho\right)  \leq L^{N}\rho^{Nd}%
\Gamma_{m},
\]
where $\Gamma_{m}$ is defined in (\ref{Def_Gamma_SAW}).
\end{lemma}

\begin{proof}
We choose $\nu^{\prime}\overset{\mathrm{def}}{=}20\nu/19$. Remark that
$\nu^{\prime\prime}\overset{\mathrm{def}}{=}\nu^{\prime}+1/100<6/5$, and
therefore $2\nu^{\prime\prime}/3<4/5.$

The assumption (\ref{Assumption}) implies%
\begin{equation}
\xi_{n}^{\left(  N\right)  }\left(  G,\rho\right)  \leq\xi_{n}^{\left(
N\right)  }\left(  \phi^{\left(  \nu\right)  },\rho\right)  ,
\label{lacebound1}%
\end{equation}
where $\phi^{\left(  \nu\right)  }=\left\{  \phi_{\nu t}\right\}  .$

We first want to get rid of the $\mathbb{I}_{\rho}$. In $\Xi_{\ell}\left(
\phi^{\left(  \nu\right)  },\rho\right)  \left(  x\right)  ,$ if all the
$m_{i}$ are $\geq1,$ we can simply use $\phi_{m\nu}\left(  x\right)  \leq
L\phi_{m\nu^{\prime}}\left(  x^{\prime}\right)  $ for $\left\vert x-x^{\prime
}\right\vert \leq\rho\leq1$ from which we easily get%
\[
\Xi_{\ell}\left(  \phi^{\left(  \nu\right)  },\rho\right)  \left(  x\right)
\leq L^{N}\rho^{Nd}\Xi_{\ell}\left(  \phi^{\left(  \nu^{\prime}\right)
},0\right)  \left(  x\right)  .
\]
There is however a complication due to the possibility of having $m_{i}=0$ in
the summation. Such $i$ have to be odd, and the possibility is not present for
$m_{1}$ and $m_{2N-1}.$ Using the fact that if $m_{i}=0$ then $m_{i-1}%
,m_{i+1}\geq1,$ we get%
\begin{equation}
\Xi_{\ell}\left(  \phi^{\left(  \nu\right)  },\rho\right)  \left(  x\right)
\leq L^{N}\rho^{Nd}\Xi_{\ell}\left(  \phi^{\left(  \nu^{\prime}\right)
},0,\mathbf{t}^{\left(  0\right)  }\right)  \left(  x\right)
\label{lacebound2}%
\end{equation}
for all $\ell\in\mathcal{L}^{\left(  N\right)  }\left[  0,n\right]  $ where
$t_{i}^{\left(  0\right)  }=0$ for $i$ even and $i=1,\ 2N-1,$ and
$t_{i}^{\left(  0\right)  }=1/200$ for the other $i$ odd. Actually, the adding
of the constant $1/200$ would be necessary only if $m_{i}$ in fact equals $0,$
but there is no harm adding it always with those $i$ for which $m_{i}$
\textit{can }be $0.$ It remains to estimate%
\[
\xi_{m}^{\left(  N\right)  }\left(  \phi^{\left(  \nu^{\prime}\right)
},0,\mathbf{t}^{\left(  0\right)  }\right)  =\sum_{\mathbf{m\in}%
\mathcal{L}^{\left(  N\right)  }\left[  0,n\right]  }\int dx_{1}\cdots
dx_{N-2}\Phi_{N,\mathbf{m},\mathbf{t}^{\left(  0\right)  }}^{\left(
\nu^{\prime}\right)  }\left(  \mathbf{x}\right)
\]
with $x=x_{N-1}$.

For $N=2$, there is $t_{i}^{\left(  0\right)  }=0$ for all $i=1,2,3$ and no
integration:%
\begin{align}
\xi_{m}^{\left(  2\right)  }\left(  \phi^{\left(  \nu^{\prime}\right)
},0\right)   &  \leq6\sum_{1\leq k\leq l\leq j,\ k+l+j=m}\phi_{k\nu^{\prime}%
}\phi_{l\nu^{\prime}}\phi_{j\nu^{\prime}}\label{Three_leg}\\
&  \leq Lm^{-d/2}\sum_{k=1}^{\left\lceil m/3\right\rceil }k^{-d/2+1}\phi
_{k\nu^{\prime}}\leq Lm^{-d/2}\sum_{k=1}^{\left[  m/2\right]  }k^{-d/2+1}%
\phi_{4k/5}\leq L\Gamma_{m}.\nonumber
\end{align}

For $N\geq3$, we apply Lemma \ref{Le_reduce}. Starting with $\mathbf{r}%
^{\left(  0\right)  }$ and $\mathbf{t}^{\left(  0\right)  }$, we recursively
define $\mathbf{r}^{\left(  k+1\right)  }\overset{\mathrm{def}}{=}%
\mathbf{r}^{\left(  k\right)  \prime},\ \mathbf{t}^{\left(  k+1\right)
}\overset{\mathrm{def}}{=}\mathbf{t}^{\left(  k\right)  \prime}.$ Applying the
lemma $N-2$ times we arrive at%
\[
\xi_{m}^{\left(  N\right)  }\left(  \phi^{\left(  \nu^{\prime}\right)
},0,\mathbf{t}^{\left(  0\right)  }\right)  \left(  x\right)  \leq L^{N-2}%
\sum_{\mathbf{m}\in\mathcal{L}_{\mathbf{r}^{\left(  N-2\right)  }}^{\left(
2\right)  }\left[  0,m\right]  }\Phi_{2,\mathbf{\mathbf{m}},\mathbf{t}%
^{\left(  N-2\right)  }}^{\left(  \nu^{\prime}\right)  }\left(  x\right)  .
\]
(There is no integration left when $N=2$). The $\mathbf{\hat{\tau}}^{\left(
N\right)  }\overset{\mathrm{def}}{=}200\mathbf{t}^{\left(  N-2\right)
},\ \mathbf{\hat{r}}^{\left(  N\right)  }\overset{\mathrm{def}}{=}%
\mathbf{r}^{\left(  N-2\right)  }$ can easily be computed: $\mathbf{\hat{r}%
}^{\left(  2\right)  }=\left(  1,1,1\right)  ,\ \mathbf{\hat{r}}^{\left(
3\right)  }=\left(  0,2,2\right)  ,\ \mathbf{\hat{r}}^{\left(  4\right)
}=\left(  1,1,3\right)  ,\ \mathbf{\hat{\tau}}^{\left(  2\right)  }=\left(
0,0,0\right)  ,\ \mathbf{\hat{\tau}}^{\left(  3\right)  }=\left(
1,0,0\right)  ,\ \mathbf{\hat{\tau}}^{\left(  4\right)  }=\left(
1,1,0\right)  $, and $\mathbf{\hat{r}}^{\left(  k+3\right)  }=\mathbf{\hat{r}%
}^{\left(  k\right)  }+\left(  1,1,1\right)  $, $\mathbf{\hat{\tau}}^{\left(
k+3\right)  }=\mathbf{\hat{\tau}}^{\left(  k\right)  }+\left(  1,1,1\right)
.$ Therefore, the only case where an $\hat{r}_{i}$ can be $0$ is $N=3$. Here
one estimates by a similar expression as on the right hand side of
(\ref{Three_leg}) with the only difference that summation over $k$ starts at
$0,$ but instead of $\phi_{k\nu^{\prime}}$ one has $\phi_{k\nu^{\prime}%
+1/200}.$ However, for $k=0,$ one estimates $\phi_{1/200}\leq L\phi
_{\nu^{\prime}},$ giving an estimate like in (\ref{Three_leg}) with a
different $L$. If $N>3,$ all the $\hat{r}_{i}^{\left(  N\right)  }$ are
$\geq1$, and it is easily checked that $2\hat{r}_{i}^{\left(  N\right)  }%
\geq\hat{\tau}_{i}^{\left(  N\right)  }.$ Using that, one estimates%
\[
\phi_{k\nu^{\prime}+\hat{t}_{i}^{\left(  N-2\right)  }}\leq L\phi
_{k\nu^{\prime\prime}}%
\]
for $k\geq\hat{r}_{i}^{\left(  N-2\right)  }$, so one gets the same estimate
as in (\ref{Three_leg}) replacing $\nu^{\prime}$ by $\nu^{\prime\prime}.$ As
$\nu^{\prime\prime}<6/5,$ the argument is the same, leading to the desired estimate.
\end{proof}

\subsection{Checking condition \ref{Cond_Main} and proof of Theorem
\ref{Th_main_SAW}}

We prove that given $\varepsilon\leq1/100$, there exists $\lambda_{0}\left(
d,\varepsilon\right)  $ such that for $0<\lambda\leq\lambda_{0}\left(
d,\varepsilon\right)  $ one has $\left\vert B_{m}\right\vert \leq\Gamma_{m}$
for all $m$, where $B_{m}\overset{\mathrm{def}}{=}\Pi_{m}/\lambda c_{m},$ and
$\Gamma_{m}$ is given by (\ref{Def_Gamma_SAW}). This is proved by induction on
$m$. Below, we use the phrase \textquotedblleft for small enough $\lambda
$\textquotedblright, in the sense that \textquotedblleft small
enough\textquotedblright\ may depend on $\varepsilon$ and $d$, but on nothing else.

For $m=1$, $\Pi_{1}\left(  x\right)  =-\lambda\phi\left(  x\right)
\mathbb{I}_{\rho}\left(  x\right)  $, and as $c_{1}=1-\lambda\int_{\left\vert
x\right\vert \leq\rho}\phi\left(  x\right)  dx\geq1-\lambda,$ we have,
provided $\lambda_{0}\left(  d,\varepsilon\right)  \leq1/2,$%
\begin{equation}
\left\vert B_{1}\right\vert \leq2\mathrm{e}^{3/4}\phi_{2/5}\leq5\phi_{2/5}%
\leq\Gamma_{1}. \label{bound_B1}%
\end{equation}
So the base of the induction is proved.

Assume now that $\left\vert B_{k}\right\vert \leq\Gamma_{k}$ for $k<m$ and
define the truncated sequence $\overline{B}_{k}$ by $B_{k}$ for $k<m,$ and $0$
for $k\geq m.$ This sequence defines $\left\{  \bar{C}_{n}\right\}  $ via
(\ref{Basic_ConvolEqu}), and then $\bar{\mu}$ given by (\ref{Sequ4}), and
$\bar{A}_{n}=\bar{\mu}^{-n}\bar{C}_{n}$. Furthermore $\bar{\delta}$ is defined
by (\ref{Def_delta}). As $\left\vert \bar{\delta}-1\right\vert \leq L\lambda,$
with $L$ depending only on $d,\varepsilon$, we have%
\begin{equation}
\left\vert \bar{\delta}\left(  1+\varepsilon\right)  -1\right\vert \leq
\frac{1}{20} \label{delta_bar_bound}%
\end{equation}
if $\lambda$ is small enough. We can apply Theorem \ref{Th_main} leading to%
\begin{equation}
\left\vert \bar{A}_{n}-\bar{a}_{n}\phi_{n\bar{\delta}}\right\vert \leq
L\lambda\left[  r_{n}\phi_{n\bar{\delta}\left(  1+\varepsilon\right)
}+n^{-d/2}\sum\nolimits_{j=1}^{\left[  n/2\right]  }j\phi_{j\bar{\delta
}\left(  1+\varepsilon\right)  }\right]  . \label{Est_A_SAW}%
\end{equation}
As $\sup_{n}\left\vert \bar{a}_{n}-1\right\vert \leq L\lambda,$ we have for
small enough $\lambda$ that $\bar{a}_{n}\phi_{n\bar{\delta}}\leq\left(
3/2\right)  \left(  1+1/100\right)  ^{d/2}\phi_{n\bar{\delta}\left(
1+\varepsilon\right)  },$ and that the right hand side of (\ref{Est_A_SAW}) is
$\leq\phi_{n\bar{\delta}\left(  1+\varepsilon\right)  },$ if $\lambda$ is
small enough, so that $\bar{A}_{n}\leq K_{1}\left(  d\right)  \phi
_{n\bar{\delta}\left(  1+\varepsilon\right)  }$, where $K_{1}\left(  d\right)
\overset{\mathrm{def}}{=}1+\left(  3/2\right)  \left(  1+1/100\right)
^{d/2},$ and therefore%
\begin{equation}
\bar{C}_{n}\leq K_{1}\left(  d\right)  \bar{\mu}^{n}\phi_{n\bar{\delta}\left(
1+\varepsilon\right)  }. \label{bound_Cbar}%
\end{equation}
As $\bar{B}_{k}=B_{k}$ for $k<m,$ we have $\bar{C}_{n}=C_{n}$ for $n<m.$

With the estimate (\ref{bound_Cbar}), we can bound $\Pi_{m}.$%
\[
\Pi_{m}^{\left(  1\right)  }\left(  x\right)  =\mathbb{I}_{\rho}\left(
x\right)  \int\prod_{\substack{0\leq s<t\leq m,\\st\neq0m}}\left(  1-\lambda
U_{st}\left(  \mathbf{x}\right)  \right)  \Phi\left[  0,n\right]  \left(
\mathbf{x}\right)  \prod_{i=1}^{n-1}dx_{i}.
\]
We bound the product inside the integral from above by dropping all bonds with
$t=m$ leading to%
\begin{align*}
\Pi_{m}^{\left(  1\right)  }\left(  x\right)   &  \leq\mathbb{I}_{\rho}\left(
x\right)  \left(  \phi\ast C_{m-1}\right)  \left(  x\right) \\
&  \leq K_{1}\left(  d\right)  \mathbb{I}_{\rho}\left(  x\right)  \bar{\mu
}^{m-1}\phi_{\left(  m-1\right)  \bar{\delta}\left(  1+\varepsilon\right)
+1}\left(  x\right) \\
&  \leq K_{1}\left(  d\right)  \bar{\mu}^{m-1}\mathbb{I}_{\rho}\left(
x\right)  \phi_{\left(  m-1\right)  \bar{\delta}\left(  1+\varepsilon\right)
+1}\left(  0\right)  .
\end{align*}

As $\left(  m-1\right)  \bar{\delta}\left(  1+\varepsilon\right)  +1\geq m/2$,
by (\ref{delta_bar_bound}), $\mathbb{I}_{\rho}\left(  x\right)  \leq\left(
4\pi/5\right)  ^{d/2}\mathrm{e}^{5/4}\phi_{2/5}\left(  x\right)  $, by
$\rho\leq1$, and $\bar{\mu}\geq1/2$, by (\ref{delta_bound}), if $\lambda$ is
small enough, we get%
\[
\Pi_{m}^{\left(  1\right)  }\left(  x\right)  \leq K_{2}\left(  d\right)
\bar{\mu}^{m}m^{-d/2}\phi_{2/5}\left(  x\right)  \leq\frac{1}{4}\bar{\mu}%
^{m}\Gamma_{m}\left(  x\right)  .
\]
with $K_{2}\left(  d\right)  \overset{\mathrm{def}}{=}2\mathrm{e}^{5/4}%
K_{1}\left(  d\right)  $, the second inequality from the way $K$ is chosen in
(\ref{K_choice}).

For $\Pi_{m}^{\left(  N\right)  }$ with $N\geq2,$ we use (\ref{Monotonicity}),
(\ref{bound_Cbar}), and Lemma \ref{Le_main_induction} and obtain $\Pi
_{m}^{\left(  N\right)  }\leq K_{1}\left(  d\right)  ^{N}\bar{\mu}^{m}%
\Gamma_{m},$ and therefore,%
\begin{align*}
\left\vert \Pi_{m}\right\vert  &  \leq\left[  \frac{\lambda}{4}+\sum
\nolimits_{N=2}^{\infty}\left(  K_{1}\left(  d\right)  \lambda\right)
^{N}\right]  \bar{\mu}^{m}\Gamma_{m}\\
&  \leq\frac{\lambda}{2}\bar{\mu}^{m}\Gamma_{m},
\end{align*}
if $\lambda$ is small enough, implying%
\begin{equation}
\left\vert B_{m}\right\vert \leq\frac{\bar{\mu}^{m}}{2c_{m}}\Gamma_{m}.
\label{first_Bm}%
\end{equation}
It remains to bound $\bar{\mu}^{m}/c_{m}.$ Remark that by (\ref{Equ_sequ_c}),
$\bar{b}_{m}=0$ and $\bar{b}_{k}=b_{k}$ for $k<m,$ we get%
\begin{align*}
\bar{c}_{m}  &  =c_{m-1}+\lambda\sum_{k=1}^{m-1}c_{k}b_{k}c_{m-k}\\
&  =c_{m-1}+\lambda\sum_{k=1}^{m}c_{k}b_{k}c_{m-k}-\lambda c_{m}b_{m}%
=c_{m}\left(  1-\lambda b_{m}\right)  .
\end{align*}
However, $\left\vert \bar{\mu}^{m}/\bar{c}_{m}-1\right\vert \leq L\lambda,$
and from (\ref{first_Bm}), we have $\left\vert b_{m}\right\vert \leq
Lm^{-d/2}\bar{\mu}^{m}/c_{m}.$ Using this, we get $\left\vert b_{m}\right\vert
\leq L$, and from that $\left\vert \bar{\mu}^{m}/c_{m}-1\right\vert \leq
L\lambda,$ so we have $\bar{\mu}^{m}/c_{m}\leq2$ for $\lambda$ small enough.
This shows that%
\begin{equation}
\left\vert B_{m}\right\vert \leq\Gamma_{m}. \label{second_Bm}%
\end{equation}

\bigskip

{\small Luca Avena, Institute of Mathematics, University of Z\"{u}rich.
luca.avena@math.uzh.ch}

{\small Erwin Bolthausen, Institute of Mathematics, University of Z\"{u}rich.
eb@math.uzh.ch}

{\small Christine Ritzmann. christine.ritzmann@gmx.ch}
\end{document}